\documentclass[a4paper]{amsart}
\usepackage{geometry}
\usepackage{graphicx}
\usepackage{amsmath,amssymb,amsthm}
\usepackage{enumerate}
\usepackage{color}
\usepackage{import}
\usepackage{hyperref}
 \usepackage[active]{srcltx}
 \usepackage[final]{pdfpages}
\usepackage[T1]{fontenc}
\usepackage{color}
\newtheorem{theo}{Theorem}

\newtheorem{lem}{Lemma}[section]
\newtheorem{defi}{Definition}[section]

\newtheorem{ass}{Assumption}[section]
\newtheorem{rmk}{Remark}[section]
\newcommand{\eps}{\varepsilon}

\newcommand{\R}{\mathbb{R}}
\newcommand{\N}{\mathbb{N}}

\renewcommand{\o}{\overline}

\numberwithin{equation}{section}

\def\debaixodaseta#1#2{\mathrel{}\mathop{\longrightarrow}\limits^{#1}_{#2}}
\def\debaixodasetafraca#1#2{\mathrel{}\mathop{\rightharpoonup}\limits^{#1}_{#2}}

\title[Competitive contaminant transport]{A degenerate elliptic-parabolic system arising in competitive contaminant transport}

\author[Ba\'\i a]{Margarida Ba\'\i a}
\address{CAMGSD/Departamento de Matem\'atica, Instituto Superior T\'ecnico, Universidade  de Lisboa, Av. Rovisco Pais 1, 1049-001 Lisboa, Portugal}
\email{mbaia@math.tecnico.ulisboa.pt}
\author[Bozorgnia]{Farid Bozorgnia}
\address{CAMGSD, Instituto Superior T\'ecnico, Universidade  de Lisboa, Av. Rovisco Pais 1, 1049-001 Lisboa, Portugal}
\email{bozorg@math.tecnico.ulisboa.pt}
\author[Monsaingeon]{L\'eonard Monsaingeon}
\address{CAMGSD, Instituto Superior T\'ecnico, Universidade  de Lisboa, Av. Rovisco Pais 1, 1049-001 Lisboa, Portugal}
\email{leonard.monsaingeon@tecnico.ulisboa.pt}
\author[Videman]{Juha  Videman}
\address{CAMGSD/Departamento de Matem\'atica, Instituto Superior T\'ecnico, Universidade   de Lisboa, Av. Rovisco Pais 1, 1049-001 Lisboa, Portugal}
\email{jvideman@math.tecnico.ulisboa.pt}

\begin{document}
\begin{abstract}
In this work we investigate a coupled system of degenerate and nonlinear partial differential equations governing the transport of reactive solutes in groundwater.  We show that the system admits a unique weak solution provided the nonlinear adsorption isotherm associated with the reaction process  satisfies certain physically reasonable structural conditions. We conclude, moreover, that the solute concentrations stay non-negative if the source 
term is componentwise non-negative and  investigate numerically the finite speed of propagation of compactly supported initial concentrations, 
in a two-component test case.
\end{abstract}
\maketitle

\section{Introduction}


Contaminant transport in groundwater and subsurface environments is a complex process given the combination of phenomena, phases and interfaces present in such environments. The transport can be influenced by an array of mechanical and chemical interactions, between or among the different constituent phases and species, such as advection, diffusion, dispersion and sorption, cf. \cite{B72,BC10}.

In this work, we will study a system of partial differential equations modeling multi-component  contaminant transport (transport of several pollutants). For expediency, we focus on the case where the transport is dominated by adsorption and diffusion. The adsorption process, adherence of a pollutant on the solid matrix at the fluid-solid interface, is assumed to be in equilibrium and of Freundlich type. In Section \ref{Motivation}, we show that these assumptions lead to   the initial-boundary value problem
\begin{equation}
\left\{
\begin{array}{ll}
\partial_t b(u)-\Delta u =f	\qquad&	 \text{in }(0,T)\times\Omega,\\
u=u_b	&	\text{on } (0,T)\times\Gamma,\\
b(u|_{t=0})=b_0& \text{in }\Omega,
\end{array}
\right.
\label{eq-IBVP}
\end{equation}
for the unknown solute concentrations  $u\!:\![0,T]\times \Omega \to \R^m$, where $f\!:\![0,T]\times\Omega \to \R^m$, $b_0\!:\!\Omega \to \R^m$ and
 $u_b\!:\![0,T]\times \Gamma \to \R^m$ are given,
$\Omega\subset \mathbb{R}^d$ is a smooth domain where the transport process takes place,  $\Gamma=\partial\Omega$, $T>0$ and  $m\geq 2$.  Moreover,  the vector field $b\!:\!\mathbb{R}^m \to \mathbb{R}^m$, associated with the Freundlich adsorption isotherm,  can be written as a gradient of a non-negative convex function  and is such that  ${\rm D}_u b(u)\rightarrow \infty$ when $u\rightarrow 0$,  i.e., the elliptic-parabolic system \eqref{eq-IBVP} becomes singular at $u= 0$. The system may thus exhibit finite speed of propagation of compactly supported initial concentrations giving rise to free boundaries that separate the regions where the solute concentrations vanish from those where they are positive.




Problem \eqref{eq-IBVP} falls under the general quasi-linear elliptic-parabolic systems addressed by Alt and Luckhaus  in their seminal paper
\cite{AL83}.  Even if our method of proof is  largely influenced by these results,
  we wish to present here a detailed, and at the same
time more straightforward, existence proof. Our approach, based on Rothe's method (\cite{R30}) and on solving a convex 
minimization problem at each time step,  provides also a simple method for the numerical approximation of solutions to system   \eqref{eq-IBVP}.

If the nonlinear term $b$  behind the time derivative in \eqref{eq-IBVP} is invertible, as a function from $\mathbb{R}^m$ to  $\mathbb{R}^m$, one can make
a  change of variables and pass the degeneracy to the diffusion term. In that case,  system  \eqref{eq-IBVP}  reduces to a coupled system of
nonlinear diffusion equations of porous medium type or, in the one-component case, to a generalized porous medium equation. This inversion is possible 
for most adsorption isotherms, thus in the one-component case the existence theory  follows from nowadays well-known results on nonlinear diffusion
equations (cf. \cite{Va07,DK07}).  We note, however, that most numerical methods applied to the single species case keep the equation 
in the form  \eqref{eq-IBVP}, see  \cite{DDW94,DDG96, BK1,BK2,ADCC01}. 

The multi-component case is much trickier, as systems of nonlinear PDEs often are, and the only existence proof  known to us leans  heavily 
on a special structure of the system wherein the sum of the solute concentrations solve a scalar generalized porous medium equation, 
see \cite{KMV16}. 
Our system \eqref{eq-IBVP}  does not possess such structure but on the other hand $b$ satisfies the assumption of being a gradient 
of scalar function which the system studied in \cite{KMV16} did not, so the existence results are mutually exclusive. At the same time, we note that our numerical 
results for the two-component case are qualitatively very similar to the ones presented in \cite{KMV16}.

The paper is organized as follows. In Section \ref{Motivation}, we present an overview of the physical modeling of  competitive contaminant transport
in groundwater and derive our model problem. In Section \ref{Notation}, we introduce our notation and some auxiliary results needed for the proof of our main existence
result that is proved
in Section \ref{Existence}, where we start by giving a mathematically precise meaning to the solutions we are looking for 
and we introduce the structural assumptions 
imposed on the nonlinear isotherm $b$.  In Section
\ref{section:positivity}, making an additional assumption on the isotherm potential $\phi$, we show that the solution stays non-negative
provided the source data is componentwise non-negative which of course is to be expected on physical grounds.  Finally, in Section \ref{Numerics}, 
we present a simple finite-difference scheme which mimics our existence proof. We solve a two-component test problem and observe that the scheme 
captures quite well the propagation of free boundaries. We  also discuss the convergence rates of our numerical scheme in the one-component scalar
case where an explicit solution is known.

From a numerical point of view, the multi-component  contaminant transport problem has  been studied  previously   in \cite{ADCC01} and 
\cite{KMV16}. In  \cite{ADCC01}, the authors considered the (non-degenerate) Langmuir isotherm for  two-component advection-dominated contaminant 
transport using a Local Discontinuous Galerkin Scheme. In  \cite{KMV16}, the numerical  results were obtained by adapting the interface-tracking
scheme introduced, for scalar porous medium equation, in \cite{Mon16}.
The scheme presented in \cite{Mon16}
could be adapted for more general   Langmuir and  Freundlich type isotherms   (including our model problem), if the problem were first expressed 
as a system of nonlinear diffusion equations, but the finite speed of propagation property of free boundaries  discussed in  \cite{KMV16} 
can best be captured if the system has the special diagonal structure as has the problem studied  in  \cite{KMV16}.

\section{Physical modeling}\label{Motivation}
Let us briefly describe the equations governing the transport of one or more solutes through a fluid-saturated porous medium, i.e. a medium characterized by a partitioning of its
total volume into a solid phase (solid matrix)
and a void or pore space that is filled by one or more fluids, see \cite{AAR89,B72,BC10} for a more detailed description. Assume that the flow is at steady state and that the
transport is  described by advection, molecular diffusion, mechanical dispersion and chemical reaction (adsorption) between a solute
and the surrounding porous skeleton, and consider the processes on a macroscopic scale.

Let $\Omega \subset \R^d$ be the domain occupied by the porous medium and let $u=u(x,t)$ denote the solute concentration in the fluid phase.
In the one-species model (contamination by one solute) $u$  solves the equation
\begin{equation}\label{equi-model}\partial_t\, \big(\theta u+\rho s\big)+ \nabla \cdot (qu-D\nabla u)=f(u),\end{equation}
where  $\theta=\theta(x)>0$  is the
porosity,  $\rho>0$ is  the constant bulk mass density of the solid matrix, $q$ is the Darcy velocity of the flow ($qu$ is the advective water flux) and $D$ stands for the hydrodynamic dispersion matrix including both molecular diffusion and mechanical dispersion  between the solute
and the surrounding porous medium. Moreover,   $f$ models source/sink terms that may depend on $u$ and $s$ denotes the concentration of the contaminant  on the porous matrix ($\rho\,\partial_t s$ corresponds to the rate of change of concentration on the porous matrix due to adsorption or desorption).

Normally, $s$ is taken to satisfy an ordinary differential
equation of the form
\begin{equation}\label{s-ne}s_t=k\, F(u,s),
\end{equation}
where $k>0$ is the rate parameter and $F$ is the reaction rate function. In case of fast reaction (\emph{equilibrium adsorption}), we may let $k\rightarrow \infty$ in \eqref{s-ne} and, assuming that the equation $F(u,s)=0$ can be solved for $s$ we obtain
\begin{equation*}
\label{s-e}
s=\psi(u),
\end{equation*}
where
 $\psi=\psi(u)$ is called the {\it adsorption isotherm} satisfying, in general,  the monotonicity property
 \[
 \psi(0)=0\, , \qquad \psi \ {\rm is} \ {\rm strictly\  increasing \ and \ smooth \ for}\  u>0.
 \]
 Moreover, $\psi$ is said to be of Langmuir type if $\psi$ is strictly concave near $u=0$ and $\lim_{u\rightarrow 0+} \psi^\prime(u)<\infty$ and of Freundlich type if $\psi$ is strictly concave near $u=0$ and $\lim_{u\rightarrow 0+} \psi^\prime(u)= \infty$.
The standard Langmuir isotherm is given by
\begin{equation*}
\label{L-i}
\psi(u)=\frac{NKu}{1+Ku},\quad K>0 \, ,
\end{equation*}
where $N$ is the saturation concentration of the adsorbed solute, and the  archetypical Freundlich isotherm is defined by
\begin{equation*}
\label{F-i}
\psi(u)=K u^p, \quad K>0\,  , 
\end{equation*}
with the reaction rate $p$ commonly chosen in $(0,1)$; the smaller  the $p$, the higher the adsorption at low concentrations is.

  In the multi-species case (transport of several contaminants),  $u=(u_1,\ldots,u_m)$ is a vector-valued  function but the evolution of the concentration of each component  $u_i$ is still described by \eqref{equi-model}. However, the adsorption process is now competitive (different species competing for the same adsorption sites) thus making the PDE system coupled. The most common
multicomponent   isotherms are of the form
$
\psi(u)=(\psi_1(u),...,\psi_m(u)) \, ,
$
with the components $\psi_i$ given by
\begin{equation}
\displaystyle \psi_{i}(u)=\frac{N_i\,K_{i} u_i}{1+\sum_{j=1}^{m} K_{j} u_j} \qquad \qquad \qquad  \ {\rm Langmuir}\, , 
\end{equation}
\begin{equation}
\displaystyle \psi_{i}(u)=C_{i} \, \Big(\sum_{j=1}^m a_{ij} u_j\Big)^{{p_i} - 1}\, u_i \qquad \qquad {\rm Freundlich} \, ,
\label{misos}
\end{equation}
see  \cite{SRS81,GF93,ADCC01,CCCS02,SMM06,L07}. Above
$N_i, C_{i},K_{i},  p_i $ and $a_{ij}$ are non-negative constants;  $N_i$ is  the adsorption capacity of the adsorbent $i$, $K_i$ is the Langmuir
affinity constant, $C_{i}$ is the Freundlich adsorption constant, $p_i$ is the Freundlich adsorption rate and $a_{ij}$ are dimensionless 
competition coefficients describing the inhibition of species $i$ to the adsorption of species $j$. By definition $a_{ii}=1$ and $a_{ij}=0$ if 
there is no competition between the species $i$ and $j$.

The multicomponent Freundlich  isotherm \eqref{misos}, first proposed by Sheindorf et al. \cite{SRS81}, has been successfully applied to different contaminants, see \cite{SBB92,GF93,CCCS02}. However, it is not the only empirical Freundlich type (according to the definition above) isotherm that has been proposed in the literature, see, {\sl e.g.}, Hinz \cite{Hinz2001}. Here, we 
assume that the adsorption process is modelled by such a generalized Freundlich type multicomponent isotherm which can be idealized as
\begin{equation*}
\psi_i(u)=C |u|^{p-1} \, u_i  \, ,
\label{F-isotherm}
\end{equation*}
where $|\cdot|$ denotes the usual Euclidean norm in $\mathbb{R}^m$. For expediency,  we also assume that the coefficient functions in \eqref{equi-model} are all constants and that the source
term $f=(f_1,\ldots,f_m)$ does not depend on $u$.  
Ignoring advection in  \eqref{equi-model}, that is setting $q=0$, it follows  that
\begin{equation}
\theta \partial_t  \big( u_i+\rho\, C |u|^{p-1} \, u_i  \big)- D\, \Delta  u_i =f_i \, ,\quad  i=1,\ldots, m \, .
\label{simple-model}
\end{equation}
We may eliminate the constants in \eqref{simple-model} by  redefining the variables through
\[
u_i^\ast= \Big(\frac{\rho C}{\theta}\Big)^{1/(p-1)}u_i \, , \qquad t^\ast=\frac{t}{D}\, , \qquad x^\ast=\frac{\theta^{1/2}x}{D}\, , \qquad f_i^\ast= \frac{\theta}{D}\, \Big(\frac{\rho K}{\theta}\Big)^{1/(p-1)}  f_i \, .
\]
Dropping the stars yields the system of equations
\begin{equation}
\partial_t\, \big( u+  |u|^{p-1} \, u \big) - \Delta u =f \,  .
\label{final-model}\end{equation}
System \eqref{final-model} is of the form $\eqref{eq-IBVP}_1$, with $b(u)=  u+  |u|^{p-1} \, u =\nabla \phi (u)$, where
\begin{equation}
 \phi(u) =  \frac{1}{2}|u|^2+\frac{1}{p+1}\, |u|^{p+1}\, , \qquad \phi:\mathbb{R}^m\rightarrow \mathbb{R}\, ,
\label{potential}
\end{equation}
is a strictly convex function.


\section{Notation and preliminary results}
\label{Notation}

Let us briefly introduce the notation used throughout this work.
The Euclidean norm  in  $\mathbb{R}^m$ is denoted by $|\cdot|$, $C$ represents a generic positive constant,
 possibly varying from line to line, and we often write $u(t)(x)$ for $u(t,x)$.

 Given a real Banach space $X$, the (Banach) space $L^{p}(0,T; X)$  consists of
 all measurable
 functions $u:[0,T]\to X$ such that
\[
\|u\|_{L^p(0,T;X)}=\left(\int_{0}^{T} \|u(t)\|_{X}^{p}\, dt\right)^{\frac{1}{p}}
<\infty\, , \qquad  1\leq p<\infty\, ,
\]
 $L^{\infty}(0,T; X)$  is the space of all measurable
 $u:[0,T]\to X$ such that
\[
\|u\|_{L^{\infty}(0,T; X)}=\sup_{t\in (0,T)} \|u(t)\|_{X}
<\infty\, ,
\]
and the Banach space $W^{1,p}(0,T; X)$, for $1\leq p<\infty$, consists of all $u\in L^{p}(0,T; X)$ such that $\partial_{t}u$ exists in the weak sense and belongs to $L^{p}(0,T; X)$.
We recall that $u \in L^{1}(0,T; X)$ is weakly differentiable with $\partial_{t}u \in L^{1}(0,T; X)$ if and only if  for some $g\in X$ it holds
 \begin{equation}\label{b2}u(t)=g + \int_{0}^{t} \partial_{t}u(s)\, ds \quad \quad \text{a.e. \ in}\ (0, T). \end{equation}
Moreover, if $u\in W^{1,p}(0,T; X)$ then  $u$ is continuous as a function from $[0, T]$ to $X$, i.e. $u\in C([0, T]; X)$.

We write $H^k(\Omega)=W^{k,2}(\Omega)$ and $u\in H^k(\Omega;\mathbb{R}^m)$ if each component of $u:\Omega\rightarrow \mathbb{R}^m$ belongs to $H^k(\Omega)$.
The action of an element $g\in H^{-1}(\Omega; \R^m)$ on $\xi\in H^{1}_{0}(\Omega; \R^m)$ is denoted by
 $\langle g, \xi\rangle$.  When no confusion arises,
  we write $\langle g, \xi\rangle$, instead of $\langle g(t), \xi(t)\rangle$,  also for $g\in L^2(0,T; H^{-1}(\Omega; \R^m))$ and
  $\xi\in L^2(0,T; H_{0}^{1}(\Omega; \R^m))$.

  Given a measurable function $u:[0,T]\to X$ and $h>0$,  we define the backward time difference quotient  of $u$
 by
$$\partial^h_t u(t,x) =\frac{u(t,x)-u(t-h,x)}{h}\, , \qquad \partial^h_t u:[h,T]\to X \, . $$

\begin{lem}\label{ebtd}
Let $u \in W^{1,p}(0,T; X)$, $1\leq p<\infty$. For any $h>0$, it holds
 $$||\partial^h_t u||_{L^{p}(h,T; X)}\leq ||\partial_{t}u||_{L^{p}(0,T; X)}.$$
 \end{lem}

\begin{proof} Given  $u \in W^{1,p}(0,T; X)$, it follows from \eqref{b2} that
 \[
 \frac{u(t)-u(t-h)}{h}=\frac{1}{h}\int_{t-h}^{t} \partial_{t}u(s)\, ds,\quad t\in [h,T]\, .
 \]
Hence
  \[\begin{array}{rl}
\displaystyle \Big|\Big|\frac{u(t)-u(t-h)}{h}\Big|\Big|_{X}  &\displaystyle  \leq\frac{1}{h}\int_{t-h}^{t} ||\partial_{t}u(s)||_{X}\, ds
= \frac{1}{h}\int_{0}^{h} ||\partial_{t}u(s+t-h)||_{X}\, ds \\ \\
 & \displaystyle  \leq  h^{\frac{1}{q}-1}  \Big(\int_{0}^{h} ||\partial_{t}u(s+t-h)||^{p}_{X}\, ds\Big)^{\frac{1}{p}}\,
\end{array}
\]
by H\"{o}lder's inequality, with $\frac{1}{p}+\frac{1}{q}=1$.  Thus, we conclude that
  \[\begin{array}{rl}
\displaystyle  \int_{h}^{T}\Big|\Big|\frac{u(t)-u(t-h)}{h}\Big|\Big|^p_{X}\, dt  &\displaystyle  \leq
\frac{1}{h}\int_{0}^{h} \Big(\int_{h}^{T} ||\partial_{t}u(s+t-h)||^p_{X}\,dt\Big)\, ds  \\ \\
&\displaystyle \leq \frac{1}{h}\int_{0}^{h} \Big(\int_{0}^{T} ||\partial_{\tau}u(\tau)||^p_{X}\,d\tau\Big)\, ds
= \int_{0}^{T} ||\partial_{\tau}u(\tau)||^p_{X}\,d\tau \, .
\end{array}
\]
using Fubini's theorem.
\end{proof}

Let $h>0$ be a time step. For simplicity take an equidistant partition of the interval $[0 \,\, T]$ and assume  that $T/h$ is an
integer. Consider a sequence of such time intervals $\{h_k\}_{k\in \N}$ with $h_k\debaixodaseta {} {k\to \infty} 0$.
Again, for simplicity, we ommit the index $k$ and write  $h\to 0$ instead of $h_k\debaixodaseta {} {k\to \infty} 0$. Given $w\in  L^{2}(0,T; X)$, we define $w_h$
as the zeroth-order Cl\'ement quasi-interpolant of $w$. In other words, given $n\equiv n(h)\in \{\,1,...,T/h\}$ we let
\begin{equation*}
w^{(n)}=\displaystyle\frac{1}{h}\int_{(n-1)h}^{nh}w(s)\, ds
\end{equation*}
 and define $w_h\in L^{\infty}(0,T; X)$ through
\begin{equation}
w_h (t)=w^{(n)}\,\,\, \text{for}\,\,\, t\in ((n-1)h, nh],
\label{global}
\end{equation}
and, if needed, $w_{h}(t)= w^{(1)}$ for $t\in (-h, 0]$.
 \begin{lem}
Given $w\in L^{2}(0,T;X)$, then
\begin{equation*}
\label{jensen}
\|w_h\|_{L^2(0,T;X)}\leq \|w\|_{L^2(0,T;X)}.
\end{equation*}
Moreover,  assuming that
$\partial_t w \in L^{2}(0,T;X)$ $($with obvious modifications if $w_h$ is extended to $(-h,0])$, it holds
\begin{equation}\label{jensend}
\|\partial_{t}^{h}w_h\|_{L^2(h,T;X)} \leq \|\partial_{t}w\|_{L^2(0,T;X)} \, .
\end{equation}
 \begin{proof} From the Jensen's inequality, it follows that
  \begin{align*}
 \|w_h\|^2_{L^2(0,T;X)} & =\int_0^T\|w_h(t)\|^2_X\, dt = \sum _{n=1}^{T/h} h \|w^{(n)}\|^2_X
  = \sum _{n=1}^{T/h} h  \left|\left|\frac{1}{h}\int_{(n-1)h}^{nh}w(s)\, ds\right|\right|_X^2 \nonumber \\
& \leq \sum _{n=1}^{T/h} \int_{(n-1)h}^{nh} \left|\left| w(s)\right|\right|_X^2 \, ds  = \|w\|^2_{L^2(0,T;X)}  \, .
 \end{align*}

Similarly, using again the Jensen's inequality as well as Lemma \ref{ebtd}, we obtain
 \begin{align*}
 \|\partial_{t}^{h}w_h\|^2_{L^2(h,T;X)} & =\int_h^T\|\partial_{t}^{h}w_h\|^2_X\, dt
 = \sum _{n=2}^{T/h} \frac{1}{h}\, \Big\|w^{(n)}-w^{(n-1)}\Big\|^2_X  \\
 & =\sum _{n=2}^{T/h} \frac{1}{h}\, \Big\|\frac{1}{h}\int_{(n-2)h}^{(n-1)h}(w(s+h)-w(s))\,  {ds} \Big\|^2_X
 \\
 &  \leq \ \sum _{n=2}^{T/h} \int_{(n-2)h}^{(n-1)h}\left\|\frac{w(s+h)-w(s)}{h}\right\|^2_{X}\, ds
 \\
 & = \int_0^{T-h}\|\partial_t^h w(s)\|^2_X ds
 =\|\partial_t^h w\|_{L^2(0,T-h;X)}^2 \leq \|\partial_{t} w \|^2_{L^2(0,T;X)} \, .
 \end{align*}
 \end{proof}
\label{whlemma}
 \end{lem}


We recall (cf. \cite{FL}) that  a function $E: X\to [-\infty,\infty]$ is said to be
sequentially lower semicontinuous,  s.l.s.c. for short, if for
any $\{u_{n}\}$   converging (in $X$) to $u$, it holds
$$ E (u)\leq \liminf_{n\to\infty}  E (u_{n})\, ,$$
and that if $X$ is a reflexive Banach space
and $E(u)\to
\infty$ as $||u||_{X}\to \infty$, then  $ E$ is
sequentially coercive with respect to the weak topology of $X$.








Let $\Psi: \R^m \to [0,\infty]$ be the Legendre conjugate of ${\phi}$, i.e.
$$ \Psi(y)= \sup_{z\in \R^m} (z\cdot y-{\phi}(z))$$
and let $B: \R^m \to [0,\infty]$ be given by
\begin{equation}\label{Bdef}
B(z):=\Psi(b(z)).\end{equation}
Given that $\phi$ is convex and $b(z)=\nabla \phi(z)$, we have (see, {\sl  e.g}, \cite{R97})
\begin{equation}\label{B}B(z)=z\cdot b(z)-{\phi}(z),\end{equation}
or, equivalently, since $\phi(0)=0$
$$
B(z)=\int_{0}^{1} \Big(b(z)-b(sz)\Big)\cdot z\,ds, \qquad z\in \R^m.
$$



\begin{lem}
Given $\delta >0$, it holds
 \begin{equation}\label{Bdef-TL-estimate}|b(z)|\leq \delta B(z) + \max_{|w|\leq 1/\delta} |b(w)| \, \qquad \forall\, z\in \R^m\, .\end{equation}
 \label{bBlemma}
 \end{lem}
  \begin{proof} If $|b(z)|\neq 0$, it follows   from \eqref{Bdef}  that
\begin{equation*}\begin{split}
B(z)&\geq b(z)\cdot \frac{1}{\delta |b(z)|}b(z) - \phi\left(\frac{1}{\delta |b(z)|}b(z)\right)\\
&\geq \frac{1}{\delta}|b(z)| - \max_{|\zeta|= 1/\delta} \phi(\zeta)
\geq \frac{1}{\delta} |b(z)| - \frac{1}{\delta}\max_{|w|\leq 1/\delta} |b(w)|,
\end{split}\end{equation*}
where in the last inequality we have used the Mean Value Theorem. 
If $\|b(z)\|= 0$ the result is trivial.  \end{proof}

\begin{rmk}
The convexity of $\phi$ and \eqref{B} imply that
 \begin{equation}\label{B-ineq}
 B(z)-B(w)\geq \big(b(z)-b(w)\big)\cdot w  \qquad  \forall z, w\in \R^m\, . \end{equation}

\end{rmk}

\section{Problem statement and existence of a weak solution}
\label{Existence}

  As mentioned in the Introduction, our aim in this work is to study 
the  solute concentrations $u: [0,T]\times  \Omega \to \R^m$ solving the initial-boundary value problem
\begin{equation}
\left\{
\begin{array}{ll}
\partial_t b(u(t,x))-\Delta u(t,x)=f(t,x)	\qquad&	 \text{in }(0,T)\times\Omega\\
u(t,x)=u_b(t,x)		&	\text{on }(0,T)\times\Gamma\\
b(u(0,x))=b_0(x)& \text{in }\Omega
\end{array}
\right.
\label{eq:rIBVP}
\end{equation}
where $x=(x_1,\ldots, x_d)$, $\bigtriangleup= \sum_{i=1}^d \frac{\partial^2}{\partial x_i^2} $, $\partial_t = \frac{\partial}{\partial t}$, $\Omega\subset \R^d$  is a smooth domain, $\Gamma=\partial\Omega$, $0<T<\infty$, and $f: [0,T]\times \Omega \to \R^m$, $b_0: \Omega \to \R^m$ and
$u_b: [0,T]\times \Gamma \to \R^m$ are given functions.  We assume that  the vector field $b: \R^m \to \R^m$ can be expressed as a gradient of  some non-negative $C^1$-convex function $\phi: \R^m\to \mathbb{R}$, with
$b(0)=0$ and $\phi(0)=0$, and that there exists a measurable  function $u_0:\Omega \to \R^m$ such that $b(u_0)=b_0.$


To define a weak solution to problem   \eqref{eq:rIBVP}
let us assume that $u_b\in L^2\big(0,T; H^{1/2}(\Gamma; \R^m)\big)$,   $f\in L^2\big(0,T; H^{-1}(\Omega; \R^m)\big)$ and $b_0: \Omega \to \R^m$ is such that
$\Psi(b_0)\in L^1 (\Omega;[0,\infty])$ (note that from Lemma~\ref{bBlemma} this last condition implies that $b_0\in L^1 (\Omega;\R^m)$).

Let $X$ be the function space defined through
\begin{equation*}\begin{array}{l}
 X=  \Big\{v\in L^2\big(0,T; H^{1}(\Omega; \R^m)\big): \\
\qquad \quad  b(v)\in L^{\infty}(0,T;L^{1}(\Omega; \R^m)),\,\, \partial_t b(v)\in L^2\big(0,T; H^{-1}(\Omega; \R^m)\big)\Big\}.
\end{array} \end{equation*}
\begin{defi}  We   say that $u\in X$ is a {\it weak solution} of  \eqref{eq:rIBVP} if

\begin{itemize}
\item[\bf a)] $u|_{(0,T)\times \Gamma}=u_b$;

\item[\bf b)] $\forall \xi\in L^2(0,T; H_{0}^{1}(\Omega; \R^m)) $ it holds
\begin{equation*}
 \int_{0}^{T}\big\langle\partial_t b(u),\xi\big> \, dt + \int_{0}^{T} \int_{\Omega} \nabla u \cdot \nabla \xi \, dx\, dt= \int_{0}^{T}   \langle f,
 \xi \rangle \, dt  \, ;
\end{equation*}

\item[\bf c)] $\forall \xi \in L^2\big(0,T; H_{0}^{1}(\Omega; \R^m)\big)\cap W^{1,1}\big(0,T;L^{\infty}(\Omega; \R^m)\big)$ with
 $\xi (T)=0$, we have
\begin{equation*}
\int_{0}^{T} \langle \partial_{t}b(u),\xi\rangle\, dt + \int_{0}^{T}\int_{\Omega} b(u) \cdot \partial_t \xi\, dx\, dt+
\int_{\Omega} b_0 \cdot \xi(0)\, dx\, =0 \, .
\end{equation*}

 \end{itemize}
\label{weaksol}
 \end{defi}
Let us  impose the following growth conditions on the vector field $b$:

\begin{enumerate}




   \item[($H_1$)] There exists $C>0$ such that
 \begin{equation*}
  |b(u)|\leq C(|u|+1)\, .
 \end{equation*}

  \item[($H_2$)] There exists $C>0$ such that
 \begin{equation*}
  |b(u)|^2\leq C(B(u)+1)\, .
 \end{equation*}

 \end{enumerate}

\begin{rmk}
The nonlinear isotherm $b(u)= u +|u|^{p-1}u$ associated with the competitive contaminant  transport potential \eqref{potential} satisfies the growth conditions  $(H_1)$ and $(H_2)$. Note that in that case $B(u)= \frac{1}{2} |u|^2+\frac{p}{p+1} |u|^{p+1}$.
\end{rmk}

We will use  Rothe's method (see \cite{Rou05}) to approximate  the initial-boundary value problem \eqref{eq:rIBVP}. It means that we will first discretize 
the time variable
and consider a nonlinear elliptic system at each time step. Next, we introduce a convex energy functional and show that it admits a unique minimizer
which also solves the elliptic system. Finally, to
 conclude that the sequence of semi-discretized solutions converges to a solution of the original problem,
we use a compactness argument based on suitable a priori estimates for the approximated solutions. The hardest part is to prove compactness
in time; here we follow the approach in \cite{AL83}.

Let $f_h$ and $u_{h,b}$ be the zeroth-order Cl\'ement quasi-interpolants of the functions $f$ and
$u_b$, respectively (see \eqref{global}).
Let $u^{(1)} \in H^1(\Omega;\mathbb{R}^m)$  be such that  for all $\xi\in H^{1}_0(\Omega; \R^m)$  it holds
\begin{equation}\label{u1}\int\limits_{\Omega}\frac{b(u^{(1)}(x))-b_0(x)}{h}\cdot \xi(x)\, dx+\int\limits_{\Omega}
\nabla u^{(1)}(x)\cdot\nabla\xi (x)\, dx=
 \langle f^{(1)}, \xi\rangle \end{equation}
 and $u^{( 1)}|_{\Gamma}={u^{(1)}_{b}}$.
In general, for $n=2,...,T/h$, and given $u^{(n-1)} \in H^1(\Omega; \R^m)$,  let $u^{(n)} \in H^1(\Omega; \R^m)$ be a solution to the problem
  \begin{equation}\label{un}\int\limits_{\Omega}\frac{b(u^{(n)}(x))-b(u^{(n-1)}(x))}{h}\cdot \xi(x)\, dx+\int\limits_{\Omega}
   \nabla u^{(n)}(x)\cdot\nabla\xi (x)\, dx=
 \langle f^{(n)}, \xi\rangle\end{equation}
 for all $\xi\in H^{1}_0(\Omega; \R^m)$, with $u^{(n)}|_{\Gamma}={u^{(n)}_{b}}$.
 The existence of  such functions $u^{(n)}$ is established in the following
lemma (see also Remark
 \ref{explain-un}).

 \begin{lem} \label{approx-min-problem}
Let $h>0$ and assume that $\o{u}\in L^2(\Omega; \R^m)$,  $g\in H^{-1}(\Omega; \R^m)$, $w\in H^{1/2}(\Gamma; \R^m)$
 and  that hypothesis $(H_1)$ holds.  Then the variational problem: find $u\in H^{1}(\Omega; \R^m)$ such that $u|_{\Gamma}=w$ and
\begin{equation}
  \int\limits_{\Omega}\frac{b(u(x))-b(\o{u}(x))}{h}\cdot \xi(x)\, dx+\int\limits_{\Omega} \nabla u(x)\cdot\nabla\xi (x)\, dx=
 \langle g, \xi\rangle\, ,
 \label{eq:weak_formulation_difference_quotient}
\end{equation}
\noindent for all $\xi\in H^{1}_0(\Omega; \R^m)$, admits at least one solution.
\label{prop:exist_time_difference_quotient}
\end{lem}

 \begin{proof}   Let     $E: H^1_0(\Omega; \R^m) \to [-\infty,\infty]$ be a functional defined through
\[\begin{array}{l}
\displaystyle E(v)=\int\limits_{\Omega}\frac{\phi(v(x)+\o{w}(x))-b(\o{u}(x))\cdot (v(x)+\o{w}(x))}{h}\, dx \\
\displaystyle \qquad \qquad + \,\frac{1}{2}\int\limits_{\Omega} |\nabla (v(x)+\o{w}(x))|^2\, dx-
 \langle g, v+\o{w}\rangle \, ,
\end{array}
\]
where $\o{w}\in H^1(\Omega;\mathbb{R}^m)$ is such that  $\o{w}$ coincides with $w$ in the trace sense on $\Gamma$.
Observe that   hypothesis $(H_1)$, the fact that $b=\nabla \phi$,  and the assumption on $\o{u}$
guarantee that $|E(v)|< \infty$ for all  ${v}\in H^1_0(\Omega; \R^m)$.
Moreover, if $v$ is a smooth critical point
of $E$  then  $u=v+\o{w}$ satisfies the Euler-Lagrange equation \eqref{eq:weak_formulation_difference_quotient} with $u|_{\Gamma}=w$.

Recalling that $\phi$ is a non-negative function,  the
  Poincar\'{e}, Cauchy-Schwarz and Young inequalities yield
$$E(v)\geq C\left[\|v\|_{H^1_0}^{2} - \|v\|_{H^1_0}-1\right] \, ,$$
for some $C=C(h,\|\o{u}\|_{L^2},\|\o{w}\|_{H^1},\|g\|_{H^{-1}})$ which shows that $E$ is coercive and thus   minimizing sequences  $\{v_k\}$ are bounded in $H^1_0$.
 In addition,  the functional $E$ is convex and  lower semicontinuous, thus  weakly lower semicontinuous, and $H^1_0(\Omega)$ is a reflexive Banach space
 so there exists a subsequence converging to an element $v$ (weakly in $H^1_0$ and strongly in $L^2$) which minimizes $E$.
Moreover,  $u=(v+\o{w})\in H^1(\Omega; \R^m)$ is a solution to \eqref{eq:weak_formulation_difference_quotient} with $u|_{\Gamma}=w$.
\end{proof}

\begin{rmk}\label{explain-un}
The existence of a solution $u^{(n)}$ to problem \eqref{un}, for $n\in \{2,...,T/h\}$, follows from Lemma  \ref{approx-min-problem} with $g=f^{(n)}$ and
$w=u_{b}^{(n)}$,  while the existence of a solution $u^{(1)}$ to problem \eqref{u1} can be established by replacing
  hypothesis $(H_1)$ by $(H_2)$ in the previous lemma. Note that $(H_2)$ implies $b_0\in  L^2(\Omega; \R^m)$.
\end{rmk}

 \begin{rmk} The solution in Lemma \ref{approx-min-problem} is unique since $\phi$ is strictly convex and thus $E$ is strictly convex.

 \end{rmk}

Let us now construct  the approximated solution of \eqref{eq:rIBVP} globally for all $t\in[0,T]$.
 Let ${u}_h \in L^{\infty}(0,T; H^{1}(\Omega))$ be the (piecewise constant in time) function defined by
\begin{equation}
\label{eq:def_interpolation}
{u}_h(t,x):=u^{(n)}(x)\quad\text{for }(t,x)\in((n-1)h,nh]\times \Omega,\,\,\,  n\in \{1,...,T/h\},
\end{equation}
where $u^{(n)}$ is the solution to problem \eqref{u1} if $n=1$ or  to problem \eqref{un} for $n\in \{2,...,T/h\}.$
We  extend ${u}_h(t)$ to  $(-h,0]$ by $u_{0}$.

By construction, $u_h$ satisfies the following result.

 \begin{lem}\label{sol-app-bvp} Let $h>0$ and assume that hypotheses $(H_1)$ and $(H_2)$ are valid. Then,
 for all $t\in (0,T)$ and all $\xi\in H^{1}_{0}(\Omega)$, the function ${u}_h$ satisfies
\begin{equation}\label{30}
  \int_{\Omega} \partial^{h}_{t}  b(u_h(t))\cdot \xi \, dx  +  \int_{\Omega} \nabla u_{h}(t) \cdot \nabla \xi \, dx=
 \langle f_h(t),
 \xi \rangle.
\end{equation}
 Moreover, $u_h|_{(0,T)\times \Gamma}= u_{h,b}$ and  for  all  $\xi \in L^2\big(0,T; H_{0}^{1}(\Omega; \R^m)\big)$ it holds
  \begin{equation*}
 \int_{0}^{T} \int_{\Omega} \partial^{h}_{t}  b(u_h(t))\cdot \xi(t) \, dx \, dt +   \int_{0}^{T}\int_{\Omega} \nabla u_{h}(t) \cdot \nabla \xi(t) \, dx\, dt=
 \int_{0}^{T} \langle f_h(t),
 \xi(t) \rangle \, dt \, .
\end{equation*}

\end{lem}
 \noindent
We also have the  following "integration by parts" formula.
  \begin{lem}\label{sol-app-bvp-1} Let  ${u}_h\in L^\infty\big(0,T;H^1(\Omega)\big)$ and assume that conditions  $(H_1)$ and $(H_2)$ are valid.  Then
\begin{align*}
\displaystyle   \int_{0}^{T}\int_{\Omega} \partial^{h}_{t}  b(u_h(t))\cdot \xi(t) \, dx\, dt
  = - \int_{0}^{T-h}  \int_{\Omega}  b(u_h(t)) \cdot \partial^{h}_{t}  \xi(t+h) \, dx\, dt
\\
\displaystyle  +\  \frac{1}{h} \left(\int_{T-h}^{T}\int_{\Omega}  b(u_h(t))\cdot \xi(t) \, dx\, dt -
   \int_{0}^{h}\int_{\Omega}  b_{0}\cdot \xi(t) \, dx\, dt\right)\, ,
\end{align*}
 for all $\xi\in L^2\Big(0,T; H_{0}^{1}(\Omega; \R^m)\Big) \cap W^{1,1}\Big(0,T; L^{\infty}(\Omega; \R^m)\Big)$.
 \end{lem}
\begin{proof} It is easy to see that
\[\begin{array}{l}
\displaystyle   \int_{0}^{T}\int_{\Omega} \partial^{h}_{t}  b(u_h(t))\cdot \xi(t) \, dx\, dt
  = \sum_{n=1}^{T/h}  \int_{(n-1)h}^{nh} \, \int_\Omega  \frac{b(u^{(n)})-b(u^{(n-1)})}{h}\cdot \xi(t)\, dx\, dt \\\\
  \qquad = - \,  \displaystyle   \sum_{n=1}^{T/h-1}  \int_{(n-1)h}^{nh} \, \int_\Omega b(u^{(n)})\cdot \frac{\xi(t+h)-\xi(t)}{h}\, dx\, dt \\ \\
\displaystyle   \qquad \qquad +\frac{1}{h}\, \bigg(\  \int_{T-h}^T \int_\Omega b(u^{(n)})\cdot\xi(t)\, dx\, dt
  -  \int_0^h \int_\Omega b(u^{(0)})\cdot\xi (t)\, dx\, dt \bigg)
  \\ \\
 \displaystyle   \qquad = -\, \int_0^{T-h} \int_\Omega b(u_h(t)\cdot \partial_t^h \xi(t+h)\, dx\, dt
 \\ \\
  \displaystyle \qquad \qquad +\ \frac{1}{h} \bigg(\int_{T-h}^T \int_\Omega  b(u_h(t))\cdot  \xi(t)\, dx\, dt -    \int_{0}^h \int_\Omega b_0\cdot \xi (t)\, dx\, dt \bigg)\, .
\end{array}
\]
\end{proof}

We will now present a series of auxiliary lemmas needed in proving the main existence result.
We will start by establishing an a priori estimate for smooth solutions. For that we need the following regularity assumption on the boundary data.
\begin{ass} Assume that the extension of $u_b\in L^2(0,T;H^{1/2}(\Omega;\mathbb{R}^m))$  to $\Omega$, still denoted by $u_b$, is such that
 \[
 u_b\in L^{\infty}(0,T; H^1(\Omega; \R^m))\,  \quad {\rm and} \quad   \partial_{t} u_b\in
   L^2(0,T; L^{2}(\Omega; \R^m))\, .
   \]
   \label{ubass}
\end{ass}

\begin{lem} Assume that problem \eqref{eq:rIBVP}  admits a smooth solution and that Assumption \ref{ubass} is valid. Then  the following a priori estimate holds
\begin{equation}\begin{array}{l}
\|B(u)\|_{L^{\infty}(0,T;L^1)}+ \|u\|_{L^{2}(0,T; H^{1})} \leq  C \big(1+\|\Psi(b_0)\|_{L^1}
\\  \\
 \qquad \qquad +\, \|f\|_{L^{2}(0,T;H^{-1})} +
 \|u_b\|_{L^{2}(0,T;H^1)} + \|\partial_ t u_b\|_{L^{2}(0,T;L^2)})\, ,
\end{array}
 \label{apriori}
 \end{equation}
for some $C=C(\Omega,T)$.
\label{smoothap}
\end{lem}

\begin{proof}
Multiplying the first equation in \eqref{eq:rIBVP} by $u-u_b$ and integrating over $\Omega$,
gives
\begin{equation}\begin{array}{l}
\displaystyle \int_{\Omega}\partial_t b(u)\cdot (u-u_b)\, dx+\int_{\Omega} |\nabla  u|^2 \, dx =
 \int_{\Omega} \nabla  u \cdot\nabla  u_b\, dx + \langle  f, u-u_b\rangle \, .
 \end{array}
 \label{aprioriaux1}
 \end{equation}
On the other hand, from  \eqref{B} it follows that
\begin{equation}
\partial_{t}\, B(u)=\partial_{t}\, b(u)\cdot u+b\cdot \partial_t u -\nabla \phi\cdot \partial_t u = \partial_{t}\,b(u)\cdot u \, .
 \label{aprioriaux2}
 \end{equation}
Substituting \eqref{aprioriaux2} into \eqref{aprioriaux1}, yields
\begin{equation}\begin{array}{l}
\displaystyle  \int_{\Omega}\partial_t B(u)\, dx + \int_{\Omega} |\nabla u  |^2  \, dx =
\displaystyle   \int_{\Omega}\partial_t b(u)\cdot u_b\, dx + \int_{\Omega} \nabla  u \cdot\nabla  u_b\, dx
+ \langle  f, u-u_b\rangle \, .
\end{array}
 \label{aprioriaux3}
 \end{equation}
Integrating \eqref{aprioriaux3}  over $(0,t)$, recalling \eqref{Bdef} and using the initial condition $b(0,x)=b_0(x)$, we obtain
\begin{equation}\begin{array}{l}
\displaystyle  \int_{\Omega} B(u(t)) \, dx  + \int_{0}^{t} \int_{\Omega} |\nabla  u|^2 \, dx \, ds = \int_{\Omega} \Psi(b_0) \, dx  +
\int_{0}^{t}  \int_{\Omega}\partial_t b(u)\cdot u_b\, dx\, ds \\ \\
\displaystyle  \qquad \qquad + \int_{0}^{t}  \int_{\Omega} \nabla  u \cdot\nabla  u_b\, dx\, ds +
\int_{0}^{t} \langle f, u-u_b\rangle \, ds \, .
 \end{array}
 \label{aprioriaux4}
 \end{equation}

Using  H\"{o}lder's, Young's and
 Poincar\'e's inequalities,  one easily shows that
 \begin{equation}\begin{array}{l}
\displaystyle  \int_0^t\,  \int_{\Omega} \nabla  u \cdot\nabla  u_b\, dx\, ds +  \int_{0}^{t}  \langle f, u-u_b\rangle \,  ds \leq  \\
\displaystyle   \qquad \qquad C\, \Big( \|f\|_{L^2(0,T;H^{-1})}+ \|u_b\|_{L^2(0,T;H^{1})}\Big) +\epsilon\, \int_{0}^{t} \int_{\Omega} |\nabla u|^2  dx\, ds \, \, .
  \label{aprioriaux5}
\end{array}  \end{equation}
holds for all  $t\in [0,T]$ and for any $\epsilon>0$ and where $C=C(\Omega,\epsilon^{-1})$.
Moreover,  integration by parts gives
\[
 \int_{0}^{t} \int_{\Omega} \partial_t b(u) \cdot u_b\, dx\, ds = \int_{\Omega} [b(u(t))u_{b}(t)-b(u(0))u_{b}(0)]\, dx -
 \int_{0}^{t} \int_{\Omega} b(u)\cdot \partial_{t}u_b \, dx\, ds,
\]
from where, using again  H\"{o}lder's and  Young's inequalities as well as hypothesis $(H_2)$, it follows,
for all  $t\in [0,T]$ and for any $\epsilon_1>0$, that
\begin{equation}\begin{array}{l}
\displaystyle \int_{0}^{t} \int_{\Omega} \partial_t b(u) \cdot u_b\, dx\, ds \leq  C\, \Big(1+
\|\Psi(b_0)\|_{L^1} \\ \\
\displaystyle  \qquad +\, \|u_b\|_{L^\infty(0,T:L^2)} + \|\partial_t u_b\|_{L^2(0,T;L^2)} + \epsilon_1\,
\max_{t\in[0,T]}  \int_{\Omega} B(u(t)) \, dx \Big) \, ,
  \label{aprioriaux6}
 \end{array}
 \end{equation}
 where $C=C(\Omega, \epsilon_1^{-1}, T)$.
Using estimates \eqref{aprioriaux5} and \eqref{aprioriaux6}  in   \eqref{aprioriaux4} and   choosing $\epsilon>0$ small enough, we first obtain
\begin{equation}\begin{array}{l}
\displaystyle   \int_{\Omega} B(u(t)) \, dx  +   \int_{0}^{t} \int_{\Omega} |\nabla  u|^2 \, dx \, ds \leq C\, \Big(1+
\|\Psi(b_0)\|_{L^1} +\|f\|_{L^2(0,T;H^{-1})}
\\ \\
\displaystyle \qquad \qquad +\,  \|u_b\|_{L^2(0,T;H^{1})} + \|\partial_t u_b\|_{L^2(0,T;L^2)}  +
 \epsilon_1\,
\max_{t\in[0,T]}  \int_{\Omega} B(u(t)) \, dx \Big) \, .
\end{array}
  \label{aprioriaux7}
  \end{equation}
Finally, taking the supremum over $t\in [0,T]$ in \eqref{aprioriaux7}, choosing $\epsilon_1>0$ small enough and using the Poincar\'e's inequality for $u-u_{b}$, yields estimate \eqref{apriori}.
\end{proof}

The next lemmas concern the semi-discrete solution ${u}_h$ defined in \eqref{eq:def_interpolation}. We will first prove an a priori bound, uniform in $h>0$, similar to the one shown in Lemma \ref{smoothap}.


\begin{lem}\label{compact-x}  Assume that hypotheses $(H_1)$ and $(H_2)$ are valid and Assumption \ref{ubass} holds.
There exists $C=C(\Omega)$ such that for any $h>0$
\begin{equation}\begin{array}{l}
\displaystyle \sup\limits_{t\in [0,T]}\int_{\Omega}B(u_h(t))\, dx+\int_0^{T}\int_\Omega |\nabla u_h|^2 \, dx\, dt \leq
C\, \big(1+\|\Psi(b_0)\|_{L^1}
\\  \\
\displaystyle  \qquad \qquad +\, \|f\|_{L^{2}(0,T;H^{-1})} +
 \|u_b\|_{L^{2}(0,T;H^1)} + \|\partial_ t u_b\|_{L^{2}(0,T;L^2)})\, .
\end{array}
\label{eq:energy_h}
\end{equation}
\end{lem}

\begin{proof}
Given any $n\in \{1,...,T/h\}$, consider  the difference $\xi=u^{(n)}-u_b^{(n)}\in H^1_0(\Omega)$ as a test function    in \eqref{un} (or in \eqref{u1}). In view of
\eqref{B-ineq}, we obtain
for all $n\geq 1$
\[\begin{array}{l}
\displaystyle \int_\Omega \big(B(u^{(n)})-B(u^{(n-1)})\big)\, dx+h\int_{\Omega}|\nabla u^{(n)}|^2 \, dx\leq\\ \\
\displaystyle \quad  h\int_\Omega \nabla u^{(n)}\cdot \nabla u^{(n)}_b\, dx+h\left\langle f^{(n)},u^{(n)}-u^{(n)}_b\right\rangle
  +\int_{\Omega}\left(b(u^{(n)})-b(u^{(n-1)})\right)\cdot u_b^{(n)}\, dx.$$
\end{array}
\]

\noindent Using Poincar\'e's  and Young's inequalities,    we immediately see that

\begin{equation}\begin{array}{l}
\displaystyle  \int_\Omega \big(B(u^{(n)})-B(u^{(n-1)})\big)\, dx + \frac{h}{2}\, \int_{\Omega}|\nabla u^{(n)}|^2 \, dx\leq \\  \\
\displaystyle  \qquad 	 C h\left(\|\nabla u^{(n)}_b\|^2_{L^2}+\|f^{(n)}\|^2_{H^{-1}}\right)
	+\int_{\Omega}\left(b(u^{(n)})-b(u^{(n-1)})\right)\cdot u_b^{(n)}\, dx \, .
\end{array}
\label{lina}
\end{equation}
Adding inequalities
 \eqref{lina} from $n=1$ to $n=N$, letting $t=Nh$, with $N\in \{1,...,T/h\}$, and  recalling that $u_h$ is piecewise constant  in time, gives
\begin{equation}\begin{array}{l}
\displaystyle  \int_\Omega  B(u_h(t))\, dx+\frac{1}{2}\, \int_0^t \int_{\Omega}|\nabla u_h(t)|^2 \, dx\leq \int_\Omega  B(u_0)\, dx +
	\\ \\
\displaystyle \qquad  C\, \left(\|\nabla u_b\|^2_{L^2(0,T;L^2)}+\|f\|^2_{L(0,T;H^{-1})}\right)   +\sum_{n=1}^{N}\int_{\Omega}\left(b(u^{(n)})-b(u^{(n-1)})\right)\cdot u_b^{(n)}\, dx\, .
\end{array}
\label{lina1}
\end{equation}
where we have also taken Lemma \ref{whlemma} into account.

Since  for any pair of sequences $\{x_n\}_{n\in \N}$ and $\{y_n\}_{n\in \N}$ the ``summation by parts'' formula
\[
\sum\limits_{n=1}^{N}(x^{n}-x^{n-1})y^{n}=(x^{N}y^{N}-x^{0}y^{0})-h\sum\limits_{n=1}^{N}x^{n-1}\left(\frac{y^{n}-y^{n-1}}{h}\right)
\]
holds, we can write
\[
\begin{array}{l}
\displaystyle  S:= \sum_{n=1}^{N}\int_{\Omega}\left(b(u^{(n)})-b(u^{(n-1)})\right)\cdot u_b^{(n)}\, dx =  \int_{\Omega}b(u_h(t))\cdot u_{h,b}(t)\, dx \, -
 \\ \\
 \displaystyle \qquad \qquad \int_{\Omega}b(u_0)\cdot u_{h,b}(0)\, dx
-\int\limits_{0}^{t}\int\limits_{\Omega}b(u_h(s-h,x))\partial_t^h u_{h,b}(s,x)\, dx\, ds.
\end{array}
\]
Hence, using H\"{o}lder's  and Young's inequalities, Lemma \ref{whlemma} and the fact that $u_h$ is piecewise constant  in time, we get
\begin{equation}
\begin{array}{l}
\displaystyle |S|   \leq  \eps \, \left(\|b(u_h(t))\|^2_{L^2} +\|b(u_h)\|^2_{L^2(-h,t:L^2)}\right) +   \\ \\
 \displaystyle \qquad \qquad  C\left(\|u_{b}\|^2_{L^\infty(0,T;L^2)}
 +\|\partial_t^h u_{b}\|^2_{L^2(0,T;L^2)}\right)
 + \int_{\Omega}|b(u_0)|^2 dx \, .
 \end{array}
  \label{spcomp1}
\end{equation}
Moreover, in view of hypotheses $(H_1)$ and $(H_2)$
\begin{equation}
\|b(u_h(t))\|^2_{L^2}\leq C\left(1+\int_{\Omega}B(u_h(t))\, dx\right)
 \label{spcomp2}
\end{equation}
and
\begin{equation}
\begin{array}{l}
\|b(u_h)\|^2_{L^2((-h,t)\times\Omega)}
 =h\|b(u_0)\|_{L^2}^2+\|b(u_h)\|^2_{L^2(0,t,L^2)}\\ \\
\qquad 	\leq C\left(1+\|B(u_0)\|_{L^1}+\|u_h\|^2_{L^2(0,t;L^2)}\right)\\ \\
\qquad \leq  C \left(1+\|\Psi(u_0)\|_{L^1} + \|u_b\|^2_{L^2(0,T;H^1)}+\|\nabla u_h\|^2_{L^2(0,t;L^2)}\right) \, ,
 \end{array}
 \label{spcomp3}
\end{equation}
where we have also use Poincar\'e's inequality for $u_h-u_{h,b}\in H^1_0(\Omega)$ and Lemma \ref{whlemma} .

Collecting estimates \eqref{spcomp1},  \eqref{spcomp2},  \eqref{spcomp3} and using them in  \eqref{lina1} and choosing $\epsilon>0$ small enough, gives finally
\[\begin{array}{l}
\displaystyle \int_{\Omega}B(u_h(t))\, dx+\int_0^{t}\int_\Omega |\nabla u_h|^2 \, dx\, dt \leq
C\, \big(1+\|\Psi(b_0)\|_{L^1}
\\  \\
\displaystyle  \qquad \qquad +\, \|f\|_{L^{2}(0,T;H^{-1})} +
 \|u_b\|_{L^{2}(0,T;H^1)} + \|\partial_ t u_b\|_{L^{2}(0,T;L^2)})\, ,
\end{array}
\]
with some constant $C=C(\Omega)>0$, independent of $h$, from which estimate \eqref{eq:energy_h} readily follows.
\end{proof}

 \begin{rmk}\label{bd-1}
Using Poincar\'e's inequality for and Lemma \ref{whlemma}, we also obtain
\[
\|u_{h}\|_{L^{2}(0,T; H^1)}\leq C\, \big(1+\|\Psi(b_0)\|_{L^1}
+\, \|f\|_{L^{2}(0,T;H^{-1})} +
 \|u_b\|_{L^{2}(0,T;H^1)} + \|\partial_ t u_b\|_{L^{2}(0,T;L^2)})\,  .\]
\end{rmk}

The following lemma is a straightforward modification of Lemma 1.9 presented in \cite{AL83}. It will be essential in proving the main existence  theorem. Its proof, which we omit, is based on Minty's monotonicity argument and makes use of estimate \eqref{Bdef-TL-estimate} of Lemma \ref{bBlemma}.
\begin{lem}\label{lemma-1.9}
Assume that  $\{u_h\}$ converges weakly in $L^{2}(0,T; H^{1}(\Omega))$ to some
 $u\in L^{2}(0,T; H^{1}(\Omega))$ and let $\tau\in (0,T)$. Moreover, assume that there exists a constant $C>0$ such that
 $$\int\limits_{0}^{T-\tau}\int\limits_{\Omega}\big(b(u_h(s+\tau))-b(u_h(s))\big)\cdot \big(u_h(s+\tau)-u_h(s)\big)\, dx\, ds\leq C\, \max\{\tau,\sqrt{\tau}\}$$
  and that
  \[
  \int_{\Omega} B(u_{h}(t))\, dx\leq C\, , \quad \text{for all}\,\,\, 0<t<T\, .
  \]
 Then $b(u_h)$ converges to $b(u)$ in $L^{1}(0,T;L^1(\Omega))$ and $B(u_{h})$ converges to $B(u)$ almost everywhere.
\end{lem}

Lemma \ref{compact-x}  guarantees that $u_h$ satisfies the second assumption of  Lemma \ref{lemma-1.9}.
The first assumption is established in the next two lemmas.

 \begin{lem}\label{635}
Let $h>0$   and  $t_1, t_2\in (0,T)$ be given. Then for all $\xi\in H^1_0(\Omega)$   it holds
\begin{equation*}\begin{array}{l}
\displaystyle  \Bigg|\int\limits_{\Omega}[b(u_h(t_2))-b(u_h(t_1)]\cdot \xi\, dx +\int\limits_{t_1}^{t_2}\left(\int_{\Omega}\nabla u_h(s)\cdot\nabla\xi \, dx
 -\left\langle f_h(s),\xi\right\rangle\right) \,ds \Bigg| \\ \\
\displaystyle  \qquad \qquad \leq  \, \sqrt{h}\, \big( \|\nabla u_h\|_{L^2(0,T;L^2)} + \|f\|_{L^2(0,T;H^{-1})}\big)\, \|\nabla \xi\|_{L^2} \, .
\end{array}
\end{equation*}
\end{lem}
\begin{proof}
From  \eqref{un}-\eqref{u1} it follows that
\begin{equation}
\int\limits_{\Omega}\left[b\left(u^{(n)}\right)-b\left(u^{(n-1)}\right)\right]\cdot\xi\, dx+h\int\limits_{\Omega}\nabla u^{(n)}\cdot\nabla\xi\, dx=h
\left\langle f^{(n)},\xi\right\rangle.\label{521}
\end{equation}
for all $n\geq 1$. Let $N_1, N_2\in \{1,...,T/h\}$ be such that $t_1\in ((N_1-1)h,N_1 h]$ and $t_2\in ((N_2-1)h,N_2 h]$. Without loss of generality assume that $N_1<N_2.$
 Summing from $N_1+1$ to $N_2$ in \eqref{521}, gives
\[
\sum_{n=N_1+1}^{N_2}\int\limits_{\Omega}\left[b\left(u^{(n)}\right)-b\left(u^{(n-1)}\right)\right]\cdot\xi\, dx +
h \sum_{n=N_1+1}^{N_2}\int\limits_{\Omega}\nabla u^{(n)}\cdot\nabla\xi\, dx = h \sum_{n=N_1+1}^{N_2}
\left\langle f^{n},\xi\right\rangle,
\]
In view of the definition of $u_h$  and $f_h$ and observing that $u_h(t_i)=u^{(N_i)}$, $i=1,2$,  we obtain
\begin{equation}\begin{array}{rl}\label{644}
\displaystyle  \int\limits_{\Omega}[b(u_h(t_2))-b(u_h(t_1))\cdot \xi \, dx
   & \displaystyle  = \int\limits_{t_1}^{t_2}\left(\int\limits_{\Omega}-\nabla u_h(s)\cdot\nabla\xi \, dx + \left\langle f_h(s),\xi\right\rangle\right)\, ds \\ \\
  & \displaystyle -\int\limits_{t_1}^{N_1 h}\left(\int\limits_{\Omega}-\nabla u_h(s)\cdot\nabla\xi \, dx + \left\langle f_h(s),\xi\right\rangle\right)\, ds
  \\ \\
& \displaystyle  +\int\limits_{t_2}^{N_2 h}\left(\int\limits_{\Omega}-\nabla u_h(s)\cdot\nabla\xi \, dx + \left\langle f_h(s),\xi\right\rangle\right)\, ds.
\end{array}
\end{equation}
 Using H\"{o}lder's inequality and Lemma \ref{whlemma},  we can estimate
  \begin{multline}
  \label{645}
 \left|\int\limits_{t_i}^{N_i h}\left(\int\limits_{\Omega}-\nabla u_h(s)\cdot\nabla\xi \, dx + \left\langle f_h(s),\xi\right\rangle\right)\, ds\right| \\
= \int\limits_{t_i}^{N_i h}   \|\nabla u_h(s)\|_{L^2}\,\|\nabla\xi\|_{L^2} \, ds + \int\limits_{t_i}^{N_i h}
\|f_h(s)\|_{H^{-1}} \|\nabla \xi\|_{L^2} \, ds\\
\leq  \, \sqrt{h}\, \big( \|\nabla u_h\|_{L^2(0,T;L^2)} + \|f\|_{L^2(0,T;H^{-1})}\big)\, \|\nabla \xi\|_{L^2} \, , \quad i=1,2\, ,   
  \end{multline}
and the result follows from \eqref{644} and \eqref{645}.
\end{proof}

\begin{lem}\label{compact-t}

Given $\tau\in (0,T)$ there exists a constant $C=C(\Omega, T, f, u_0,u_b,\tau)$ such that for all $h>0$ it holds
\begin{equation*}
 \int\limits_{0}^{T-\tau}\int\limits_{\Omega}\left[b(u_h(t+\tau))-b(u_h(t))\right]\cdot \left[u_h(t+\tau)-u_h(t)\right]\, dx\, dt\leq  C\, \max\{\tau,\sqrt{\tau}\} \, .
\end{equation*}
 \label{prop:compactness_time}
\end{lem}

\begin{proof}  Consider first  $\tau <h$. Since $u_h(t+\tau)\equiv u_h(t)\ \forall t\in((n-1)h,nh-\tau]$, we get
\begin{align*}
&\int\limits_{0}^{T-\tau}\int\limits_{\Omega}\left[b(u_h(t+\tau))-b(u_h(t))\right]\cdot \left[u_h(t+\tau)-u_h(t)\right]\, dx\, dt  \\
&=\sum_{n=1}^{T/h-1}\int\limits_{nh-\tau}^{nh} \int\limits_{\Omega}\left[b(u_h(t+\tau))-b(u_h(t))\right]\cdot \left[u_h(t+\tau)-u_h(t)\right]\, dx\, dt\\
&=\sum_{n=1}^{T/h-1}\int\limits_{nh-\tau}^{nh} \int\limits_{\Omega}\left[b(u^{(n+1)})-b(u^{(n)})\right]\cdot \left[u^{(n+1)}-u^{(n)}\right]\, dx\, dt\\
&=\tau \sum_{n=1}^{T/h-1}\int\limits_{\Omega}\left[b(u^{(n+1)})-b(u^{(n)})\right]\cdot \left[u^{(n+1)}-u^{(n)}\right]\, dx\\
&=\tau \sum_{n=1}^{T/h-1}\int\limits_{\Omega}\left[b(u^{(n+1)})-b(u^{(n)})\right]\cdot \left[v^{(n+1)}-v^{(n)}\right]\, dx- \\
&\hspace{1cm}\tau \sum_{n=1}^{T/h-1}\int\limits_{\Omega}\left[b(u^{(n+1)})-b(u^{(n)})\right]\cdot \left[u_{b}^{(n+1)}-u_{b}^{(n)}\right]\, dx\\
&:=\tau \, (I_{1,h}+I_{2,h}) \, ,
\end{align*}
 where $v^{(n)}=u^{(n)}-u^{(n)}_b\in H^1_0(\Omega)$.
Equation \eqref{30} implies that
\[
  I_{1,h}=\sum_{n=1}^{T/h-1}h\Big[\langle f^{(n)},
v^{(n+1)}-v^{(n)}\rangle -\int\limits_{\Omega}\nabla u^{n+1}\cdot \nabla (v^{(n+1)}- v^{(n)}) \, dx\Big]  \, .
\]
From Lemma \ref{whlemma}  and  estimate \eqref{eq:energy_h} it then follows that
\[
I_{1.h}  \leq C \big( \int_0^T \|f_h\|_{H^{-1}}\|\nabla v_h\|_{L^2}\, dt + \int_0^T \|\nabla u_h\|_{L^2}\|\nabla v_h\|_{L^2}\, dt \big) \leq K_1\, ,
\]
for some constant  $K_1=K_1(\Omega, f, b_0,u_b) >0$, independent of $h$ and $\tau$.
Moreover,
 \begin{align*}
 I_{2,h} & = \frac{1}{h}\,  \int_0^{T-h} \int\limits_{\Omega}\left[b(u_h(t+h))-b(u_h(t))\right]\cdot \left[u_{h,b}(t+h)-u_{h,b}(t)\right]\, dx\, dt \\
 & \displaystyle   \leq C \, \|b(u_h)\|_{L^2(0,T;L^2)}   \|\partial_t ^h u_{h,b}\|_{L^2(0,T;L^2)} \\
 &\displaystyle    \leq C \, (1+ \|u_h\|_{L^2(0,T;L^2)} )\, \|\partial_t u_b\|_{L^2(0,T;L^2)} \leq K_2\, ,
  \end{align*}
 with some $K_2=K_2(\Omega, T, f, b_0,u_b)$ and where we have taken account hypothesis $(H_1)$ as well as   estimates \eqref{jensend} and  \eqref{eq:energy_h}.
  We thus conclude that
\[
\int\limits_{0}^{T-\tau}\int\limits_{\Omega}\left[b(u_h(t+\tau))-b(u_h(t))\right]\cdot \left[u_h(t+\tau)-u_h(t)\right]\, dx\, dt\leq  C\, \tau ,
\]
for some $K=K(\Omega, T, f, b_0,u_b,\tau)$.

Assume next  that  $\tau \geq h$ and  fix  $t\in [0, T-\tau]$.
In view of Lemma \ref{635} and estimate  \eqref{eq:energy_h},  for all $\xi(t) \in H^1_0(\Omega)$ it holds
\[
\begin{array}{l}
\displaystyle  \int\limits_{\Omega}\Big[b(u_h(t+\tau)-b(u_h(t))\Big]\cdot  \xi(t)  \, dx +
\int\limits_{t}^{t+\tau}\Big(\int\limits_{\Omega}\nabla u_h(s)\cdot\nabla\xi\, dx -
 \left\langle f_h(s),\xi(t)\right\rangle\, ds\Big)\\
\displaystyle  \qquad \qquad    \leq C \sqrt{h}\, \|\nabla \xi(t)\|_{L^2} \leq C \sqrt{\tau}\, \|\nabla \xi(t)\|_{L^2}\, ,
\end{array}
\]
where the constant $C>0$ is independent of $\tau$ and $h$.  Taking $\xi(t)= v_h(t+\tau)-v_h(t)$, where $v_h(t)=u_h(t)-u_{h,b}(t)$, and  integrating in $t$ from $0$ to $T-\tau$, we get
 \begin{align*}
 &\int\limits_{0}^{T-\tau}\int\limits_{\Omega}\left(b(u_h(t+\tau))-b(u_h(t))\right)\cdot \left(u_h(t+\tau)-u_h(t)\right)\, dx\, dt\\
  & \hspace{1cm}\leq C \sqrt{\tau\, (T-\tau)} + \int\limits_{0}^{T-\tau}\int\limits_{\Omega}\left(b(u_h(t+\tau))-b(u_h(t))\right)\cdot
  \left(u_{h,b}(t+\tau)-u_{h,b}(t)\right)\, dx\, dt\\
  & \hspace{1cm}\phantom{=}-\int\limits_0^{T-\tau}\Big(\int\limits_{t}^{t+\tau}\int\limits_{\Omega}\nabla u_h(s)\cdot\nabla\xi(t)\, dx\,ds\Big)dt
  +\int\limits_0^{T-\tau}\Big(\int\limits_{t}^{t+\tau}\left\langle f_h(s),\xi(t)\right\rangle\, ds\Big)dt\\
  & \hspace{1cm} :=  C \sqrt{\tau\, (T-\tau)}+\mathcal{I}_{1,h}+\mathcal{I}_{2,h}+\mathcal{I}_{3,h} \, ,
 \end{align*}
 where we have also used  estimate  \eqref{eq:energy_h} to bound $\|\nabla \xi\|_{L^2(0,T;L^2)}$.

The integral $\mathcal{I}_{1,h}$ can be estimated as $I_{2,h}$ in the previous step yielding
\[
      |I_{1,h}| \leq C\, \tau \, .
\]
As for  $\mathcal{I}_{2,h}$, we proceed as follows
\[\begin{array}{l}
\displaystyle |\mathcal{I}_{2,h}| \leq
\int_0^{T-\tau} \Big(\|\nabla \xi (t)\|_{L^2}\, \int_t^{t+\tau} \|\nabla u_h(s)\|_{L^2}\, ds \, \Big) dt \\   \\
\displaystyle  \leq \sqrt{\tau}\, \|\nabla u_h\|_{L^2(0,T;L^2)}\,
\int_0^{T-\tau} \|\nabla \xi (t)\|_{L^2}\, dt  \leq C\, \sqrt{\tau\, (T-\tau)}\, ,
\end{array}
\]
where we have used again  estimate  \eqref{eq:energy_h}. For the integral $\mathcal{I}_{3,h}$ we obtain similarly
$|\mathcal{I}_{3,h}| \leq    C\, \sqrt{\tau\, (T-\tau)}$,  and we thus conclude that
\[\begin{array}{l}
\displaystyle \int\limits_{0}^{T-\tau}\int\limits_{\Omega}\left(b(u_h(t+\tau))-b(u_h(t))\right)\cdot \left(u_h(t+\tau)-u_h(t)\right)\, dx\, dt
 \\ \\
 \displaystyle \leq C \big( \tau +\sqrt{\tau(T-\tau)} \big) \leq C\, \max\{\tau,\sqrt{\tau}\}\, ,
 \end{array}\]
 for some constant $C=C(\Omega, T, f, b_0,u_b)$ independent of $h$ and $\tau$.
\end{proof}

We are now in a position to prove the main existence theorem.

 \begin{theo}\label{main}
 Assume that  hypotheses $(H_1)$ and $(H_2)$ are valid and that Assumption \ref{ubass} holds. There exists a unique weak solution $u\in X$  to problem \eqref{eq:rIBVP} such that $B(u)\in L^{\infty}(0,T;L^1(\Omega;\mathbb{R}^m))$.
  \end{theo}

\begin{proof} Let us start with existence: for any  $h>0$, let ${u}_h \in L^{2}(0,T; H^{1}(\Omega))$ be the piecewise constant function
 defined in \eqref{eq:def_interpolation} and which,  in view of
Lemmas \ref{sol-app-bvp} and \ref{sol-app-bvp-1},
 satisfies
\begin{itemize}

\item[i)]  $u_h|_{(0,T)\times \Gamma}= u_{h,b}$;

\item[ii)] for  all  $\xi \in L^2\big(0,T; H_{0}^{1}(\Omega; \R^m)\big)$
\begin{equation}\label{feliz}
 \int_{0}^{T} \int_{\Omega} \partial^{h}_{t}  b(u_h(t))\cdot \xi(t) \, dx \, dt +   \int_{0}^{T}\int_{\Omega} \nabla u_{h}(t) \cdot \nabla \xi(t) \, dx\, dt=
 \int_{0}^{T} \langle f_h(t),
 \xi(t) \rangle \, dt;
\end{equation}

\item[iii)] for all $\xi\in L^2\big(0,T; H_{0}^{1}(\Omega; \R^m)\big) \cap W^{1,1}\big(0,T; L^{\infty}(\Omega; \R^m)\big)$, with $\xi(T)=0$,
\begin{equation}\begin{array}{l}\label{31}
\displaystyle   \int_{0}^{T}\int_{\Omega} \partial^{h}_{t}  b(u_h(t))\cdot \xi(t) \, dx\, dt
  = - \int_{0}^{T-h}  \int_{\Omega}  b(u_h(t))\partial^{h}_{t}  \xi(t+h) \, dx\, dt
\\ \\
\displaystyle   \qquad  +\  \frac{1}{h} \left[\int_{T-h}^{T}\int_{\Omega}  b(u_h(t))\cdot \xi(t) \, dx\, dt -
   \int_{-h}^{0}\int_{\Omega}  b_{0}\cdot \xi(t+h) \, dx\, dt\right]\, .
   \end{array}
   \end{equation}

\end{itemize}
Consider a sequence $\{u_h\}$ of such functions and note  that Lemma \ref{compact-x} (see also Remark \ref{bd-1}) guarantees the
existence of  a function ${u} \in L^{2}(0,T; H^{1}(\Omega;\mathbb{R}^m))$ such that
(up to subsequence)
\begin{equation}\label{poli}{u}_h \debaixodasetafraca {}{h\to 0} u\quad \text{in}\quad L^2(0,T; H^{1}(\Omega;\R^m)).\end{equation}

For $u$ to be a weak solution to problem \eqref{eq:rIBVP}, it needs to meet the  conditions {\bf a), b)} and {\bf c)} of Definition \ref{weaksol} and be such that
\[
b(u)\in L^\infty(0,T;L^1(\Omega;\mathbb{R}^m)) \quad \text{and} \quad \partial_t b(u)\in L^2(0,T;H^{-1}(\Omega;\mathbb{R}^m)\, .
\]
Now, $u_{h,b}$ is the (piecewise constant in time) Cl\'ement  interpolant of $u_b$ and $u_h$ satisfies \eqref{poli} so
 \[
 {u}_{h,b}-u_h \debaixodasetafraca  {}{h\to 0} u_b-u\quad \text{in}\quad L^2(0,T;H^{1}(\Omega,\R^m)) \, .
 \]
Hence, $u_b-u\in L^2(0,T;H^{1}_{0}(\Omega,\R^m))$ and condition {\bf a} holds.

To show that  {\bf b} holds, we first observe that $f_h$ being the Cl\'ement  interpolant of $f$ and \eqref{poli}  guarantee that
%
\begin{equation*} \int_{0}^{T}   \langle f_h, \xi \rangle \, dt \debaixodaseta {}{h\to 0} \int_{0}^{T}   \langle f, \xi \rangle \, dt\, , \qquad
 \int_{0}^{T}     \nabla u_h\cdot \xi   \, dt \debaixodaseta {}{h\to 0} \int_{0}^{T}     \nabla u\cdot \xi  \, dt
 \end{equation*}
for all $ \xi\in L^2(0,T; H_{0}^{1}(\Omega; \R^m))$.
 Moreover,  from \eqref{feliz},  we obtain
\[\begin{array}{l}
\displaystyle \Big| \int_0^T  \left<\partial^{h}_t b(u_h),\xi\right>\, dt \Big| = \Big| \int_{0}^{T} \int_{\Omega} \partial^{h}_{t}  b(u_h(t))\cdot \xi(t) \, dx \, dt \Big | \\ \\
\displaystyle \leq
(\|\nabla u_h\|_{L^2(0,T;L^2)}+\|f_h\|_{L^2(0,T,H^{-1})})\|\xi\|_{L^2(0,T;H^1_0)}\, ,
\end{array}
\]
 for all  $ \xi\in L^2(0,T; H_{0}^{1}(\Omega; \R^m))$.
Using Lemmas \ref{compact-x} and \ref{whlemma}, we thus conclude that
 \begin{equation}
\|\partial^{h}_{t}  b(u_h)||_{L^2(0,T;H^{-1})}\leq  C\, ,
\label{bqest}
\end{equation}
for some constant $C>0$ depending on $f,u_b,b_0$ and $\Omega$ but independent of $h$.

Estimate \eqref{bqest} guarantees the existence of $\lambda\in L^2(0,T;H^{-1}(\Omega; \R^m))$ such that
\[
\partial^{h}_{t}  b(u_h)\debaixodasetafraca {}{} \lambda
\qquad \text{in }L^2(0,T;H^{-1}(\Omega; \R^m)).
\]
To  conclude {\bf b)}, it remains to be checked that
$\lambda= \partial_{t}b(u)$ in $L^2(0,T;H^{-1}(\Omega; \R^m))$.
 To this end, it is enough to show that we can pass to the limit in \eqref{31}.
 The left-hand side immediately yields
 \[
  \int_{0}^{T}\int_{\Omega} \partial^{h}_{t}  b(u_h(t))\cdot \xi(t) \, dx\, dt =   \int_{0}^{T} \left<\partial^{h}_{t}  b(u_h),\xi\right>
\debaixodaseta {}{h\to 0}   \int_{0}^{T} \left<\lambda,\xi\right>
 \]
For the first term on the right-hand side, the estimate $\|u_h\|_{L^2(0,T;H^1)}\leq C$ from Remark~\ref{bd-1} allows to apply Lemma~\ref{lemma-1.9} and conclude that $b(u_h)\to b(u)$ strongly in $L^1(0,T;L^1(\Omega;\mathbb{R}^m))$ and, in particular, $b(u_h)\to b(u)$ pointwise a.e. in $(0,T)\times\Omega$.
Since, obviously $\partial_t^h \xi(t+h,x)=\frac{\xi(t+h,x)-\xi(t,x)}{h}\to \partial _t\xi(t,x)$ a.e. in $(0,T)\times\Omega$, and by Lemma~\ref{ebtd} $\|\partial_t^h \xi\|_{L^1(0,T;L^{\infty})}\leq C$, we have
\[
 \int_{0}^{T-h}  \int_{\Omega}  b(u_h(t))\partial^{h}_{t}  \xi(t+h) \, dx\, dt\to \int_{0}^{T}  \int_{\Omega}  b(u(t))\partial_{t}  \xi(t+h) \, dx\, dt \, .
\]
For the second term on the right-hand side  in \eqref{31}, the regularity $\xi\in W^{1,1}(0,T;L^{\infty}(\Omega;\mathbb{R}^m)$ guarantees that
\[
\frac{1}{h}\int_{T-h}^{T} \|\xi(t)\|_{L^\infty(\Omega)} dt \to \|\xi(T)\|_{L^\infty(\Omega)}=0\, ,
\]
thus
$$
 \left|\frac{1}{h}\int_{T-h}^{T}\int_{\Omega}  b(u_h(t))\cdot \xi(t) \, dx\, dt\right|
 \leq \|b(u_h)\|_{L^\infty(0,T;L^1)}\frac{1}{h}\int_{T-h}^{T} \|\xi(t)\|_{L^\infty(\Omega)} dt\to 0.
$$
Moreover, $\frac{1}{h}\int_{0}^{h} \xi(t)\,dt\to \xi(0)$ strongly in $L^\infty(\Omega)$, and thus
$$
 \frac{1}{h}\int_{0}^{h}\int_{\Omega}  b_{0}\cdot \xi(t) \, dx\, dt = \int_{\Omega} b_0\left(\frac{1}{h}\int_{0}^{h} \xi(t)\,dt\right)\,dx \to \int_\Omega b_0 \xi(0)\,dx.
$$
Therefore,  passing to the limit in \eqref{31}, we conclude that
\begin{equation*}
\label{feliz2016}
\int_{0}^{T} \langle\lambda, \xi(t)\rangle\, dt= -\int_{0}^{T}\int_{\Omega} b(u(t)) \partial_t \xi(t)\, dx\, dt - \int_{\Omega}
 b_0\cdot\xi(0) \, dx,
 \end{equation*}
 for all $\xi\in L^2\Big(0,T; H_{0}^{1}(\Omega; \R^m)\Big) \cap W^{1,1}\Big(0,T; L^{\infty}(\Omega; \R^m)\Big)$ with $\xi(T)=0$, which means of course that $\lambda =\partial_t b(u)\in  L^2(0,T;H^{-1}(\Omega;\mathbb{R}^m)$ and the initial condition $b(u(0))=b_0$ are satisfied in the weak sense. Hence, conditions {\bf b} and  {\bf c} hold.

As for the uniqueness, the proof of \cite[Theorem 2.4]{AL83} applies directly since we are considering linear diffusion in \eqref{eq-IBVP}; note that hypothesis $(H_1)$ guarantees that $b(u)\in L^2(0,T;L^2(\Omega;\mathbb{R}^m))$.

Finally, we show that $B(u)\in L^\infty(0,T;L^1(0,T;\mathbb{R}^m))$ from which, using hypotheses $(H_2)$, it follows that $b(u)\in L^\infty(0,T;L^2(\Omega;\mathbb{R}^m))$ so that $u\in X$.
\end{proof}


\section{Positive properties of the solutions}
 \label{section:positivity}

Under an additional structural assumption on the contaminant transport potential $\phi$, we retrieve nonnegative solutions provided the  source data is (componentwise) nonnegative.
More precisely, we can state the following result.
\begin{theo}
Let the assumptions of Theorem \ref{main} hold and  assume, moreover, that $\phi(u)=\Phi(|u|)$ for some convex and nondecreasing $C^1$ function $\Phi:\R^+\to \R^+$ such that $\Phi(0)=0=\min_{w\in\mathbb{R}^+}\Phi(w)$.
If the functions $\overline u$ and $g$ in Lemma \ref{approx-min-problem} are  componentwise non-negative and $g\in L^2(\Omega;\mathbb{R}^m)$, then the minimizer $u$ is (componentwise) nonnegative.
\end{theo}

\begin{rmk} We can no longer allow $g$ to be an $H^{-1}(\Omega;\mathbb{R}^m)$ distribution. Physically,
 there cannot be any sinks, i.e. the sources must be non-negative functions $f_i=f_i(t,x)\geq 0$ for all $t\geq 0$.

Note also that the structural assumption on the potential $\phi$ can be weakened to $\phi(u^+)-b(\overline u)\cdot u^+\leq \phi(u)-b(\overline{u})\cdot u$, for fixed $\overline{u}\in \R^m_+$ and for all $u\in \R^m$, which is a kind of additional convexity/monotonicity property.
\end{rmk}
\begin{proof}
It is enough to prove that, if $\overline u\geq 0$, then any minimizer $u$ from Lemma~\ref{approx-min-problem} remains (componentwise) non-negative. Indeed, iterating will give positivity $u^{(n)}\geq 0$ at the discrete level for all $n\geq 0$, thus the piecewise-constant interpolation $u_h(t)\geq 0$ and the final solution $u=\lim\limits_{h\to 0}u_h\geq 0$ as well.

Our extra assumption on the potential readily implies that
$$
\phi(u^+)=\Phi(|u^+|)\leq\Phi(|u|)=\phi(u)
$$
for all $u\in \R^m$, where $u^+$ denotes the componentwise positive part $(u^+)_i=\max(u_i,0)$ (and of course $|u^+|\leq |u|$).
Moreover for all $\overline{u}\geq 0$ we see that $b(\overline{u})=\Phi'(|\overline{u}|)\frac{\overline{u}}{|\overline{u}|}\geq 0$ has nonnegative components as soon as $\overline{u}\geq 0$.
As a consequence, it holds
$$
b(\overline{u})\cdot u \leq b(\overline u)\cdot u^+.
$$
Then using $u^+$ in the minimization problem, we get
\begin{align*}
E(u^+)	&	= \int\limits_{\Omega}\frac{\phi(u^+)-b(\overline{u})\cdot u^+}{h} \, dx +\frac{1}{2}\int\limits_{\Omega}|\nabla u^+|^2 \, dx -\int\limits_{\Omega} g\cdot u^{+} dx\\
  & \leq \int\limits_{\Omega}\frac{\phi(u)-b(\overline{u})\cdot u}{h}+\frac{1}{2}\int\limits_{\Omega}|\nabla u|^2\, dx  -\int\limits_{\Omega} g\cdot u \, dx   =E(u),
\end{align*}
where we have used Stampacchia's result $\nabla u_i^+=\chi_{[u_i>0]}\nabla u_i$ and thus $\|\nabla u^+_i\|_{L^2}\leq \|\nabla u_i\|_{L^2}$.
Uniqueness of the minimizer  implies  that $u^+=u$, from which it follows that $u(x)\geq 0$ componentwise and a.e. $x\in \Omega$.
\end{proof}

%

\section{Numerical results}
\label{Numerics}
In this section, we will complement the previous existence results by implementing a simple finite-difference algorithm to approximate the solution $u$ 
of problem \eqref{eq:rIBVP}. For the ease of exposition we restrict ourselves to the one-dimensional case ($d=1$) with two species ($m=2$).

Let us choose a fixed space interval $[a,b]$ and, for fixed $I\in \mathbb N$, we take a uniform partition
$$
\Delta x =\frac{b-a}{I+1},\qquad x_i=a+i\Delta x,\, \quad i=0\ldots I+1 \, ,$$
so that $x_{I+1}=b$. Similarly,  for  a fixed $T<\infty$ and $N\in \mathbb N$, we set
$$
\Delta t = \frac{T}{N},\qquad t^n =n\Delta T,\, \quad n=0\ldots N.
$$
The unknowns take values at the meshpoints $(x_i,t^n)$ and, as usual, we write
$$
U^n_i\approx u(x_i,t^n)\, , \quad i=1,\ldots, I\, , \ \ n=0,\ldots, N\, .
$$
For simplicity, let us consider homogeneous Dirichlet boundary conditions
$$
u(a,t)=u(b,t)=0\qquad\Leftrightarrow \qquad U^{n}_{0} = U^{n}_{I+1}=0 \ \ \forall n\, ,
$$
and assume that
$ f=0$.
We construct a global approximation,  continuous in $x$  and piecewise linear at each  interval $K_i=(x_i,x_{i+1})$, to approximate each of the terms in 
the discrete system of equations. To this end, we first choose the elementary quadrature method
$$
\int_{x_i}^{x_{i+1}}\phi\left(u(x,t^{n+1})\right)\,dx\approx \frac{\phi(U^{n+1}_i)+\phi(U^{n+1}_{i+1})}{2}\Delta x.
$$
Defining $b^{n}_i=b\left( u^{n}_i\right)$, we obtain similarly the approximation
$$
\int_{x_i}^{x_{i+1}}b\left(u(x,t^{n})\right)\cdot u(x,t^{n+1})\,dx\approx \frac{b^{n}_i\cdot U^{n+1}_i + b^{n}_{i+1}\cdot U^{n+1}_{i+1}}{2}\Delta x\, .
$$
Moreover, the Dirichlet energy is approximated as
$$
\frac{1}{2}\int_{x_i}^{x_{i+1}}|\nabla u(x,t^{n+1})|^2\,dx\approx \frac{1}{2}\left|\frac{ U_{i+1}^{n+1}- U_{i}^{n+1}}{\Delta x}\right|^2\Delta x.
$$
Following the approach presented in the previous sections, and after initialization
$$
 U^0_i= u^0(x_i)\, ,
$$
the approximate solution is determined by solving recursively
\begin{equation}
U^{n+1}=\operatorname{Argmin} F^n,
\label{discrete_approx-min-problem}
\end{equation}
where the real-valued function $F^n:(\R^m)^I\to \R$ is the spatial approximation of
$$
\mathcal{F}^n(u)=\int\limits_{\Omega}\big(\phi(u(x))- b(u^n(x))\cdot u(x)\big)\, dx \,+ \,\frac{\Delta t}{2}\int\limits_{\Omega} | \nabla u(x)|^2\, dx\, ,
$$
and, according to the above-defined quadratures, is explicitly written as
\begin{equation}
F^n( U):=\sum\limits_{i=1}^I\phi( U_i)\Delta x-\sum\limits_{i=1}^I b^n_i\cdot  U_i\, \Delta x +\frac{\Delta t}{2}\sum\limits_{i=0}^I\left|\frac{U_{i+1}- U_i}{\Delta x}\right|^2\Delta x\, .
\label{eq:discrete_energy_functional}
\end{equation}
Note that both the left and right cells, $[x_{i-1},x_i]$ and $[x_{i},x_{i+1}]$, contribute with weight $1/2$ to the first two terms for each $i$. 
Moreover, we have exploited the zero boundary conditions $U_0=U_{I+1}=0$ in the last term
 and observed that  no contributions $\phi(0)=0$ and $b(0)=0$ arise from the first two terms for $i=0,I+1$.
This construction is nothing but the spatial discretization of the function $u^{(n)}$ defined \eqref{eq:def_interpolation} using the minimization scheme of Lemma~\ref{approx-min-problem}.

 All the simulations were performed using \texttt{Gnu/Octave} on a personal computer. The numerical minimization step can be performed using any unconstrained optimization method. We have  used the \texttt{fminunc} toolbox with gradient option, since the Jacobian $D_UF^n$ can easily be 
 computed analytically from the above formula.

\begin{rmk}
 We point out that our scheme could be adapted to arbitrary space dimension using, {\sl e.g.}, finite elements.
 One could also choose higher order quadratures, but the naive trapezoidal quadrature method used here has the advantage of allowing  the Jacobians 
 to be computed analytically.
 \end{rmk}

\subsection{A two-component test-case}
As an example we present here a two-component computation $u=(u_1,u_2)$ for the Freundlich isotherm $b(u)=u+|u|^{p-1}u$, with exponent $p=1/3$ 
(for instance). The computation was performed for $\Delta t=\Delta x=10^{-2}$ in a domain $(x,t)\in (-2,2)\times(0,0.5) $ (but other intervals could have been used) and the result is shown in Fig.~\ref{fig:cauchy}.
\begin{figure}[h!]
 \begin{center}
 \def\svgwidth{1.\textwidth}
\centering
  \includegraphics[width=1.\textwidth]{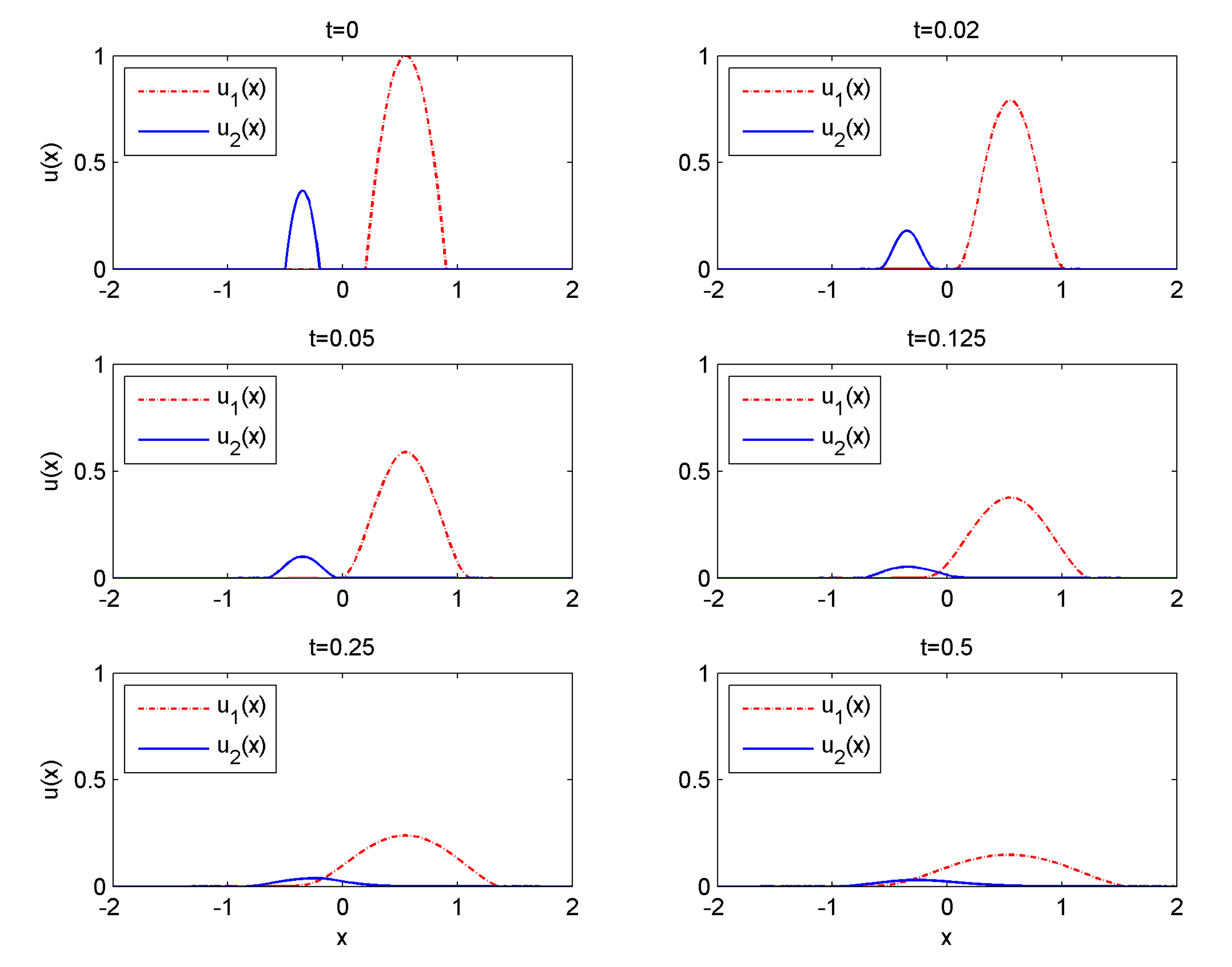}
 \end{center}
  \caption{snapshot of $u_1(t,.)$ (dashed line) and $u_2(t,.)$ (solid line) as functions of $x$ for several times from $t=0$ to $t=0.5$.}
 \label{fig:cauchy}
 \end{figure}

It is worth pointing out that our scheme seems to be automatically positive, as should be expected from the results of Section~\ref{section:positivity} since the particular isotherm $b(u)=u+|u|^{1/3-1}u$  maps the positive cone $(\R^+)^2$ into itself. In the same spirit and as already discussed, one should expect that the 
singularity/degeneracy $D_u b(0)\approx\infty$ leads to finite speed of propagation and free boundary solutions. The scheme seems to capture reasonably well the propagation of free boundaries, as shown  in Fig.~\ref{fig:cauchy}; the left and right free-boundaries stay away from the boundaries $x=\pm 2$ for small times, and the internal hole at time $t=0$ persists until $t\approx .05$ (after which the supports of $u_1(t,.)$ and $u_2(t,.)$ match for all later times and progressively invade the whole domain with finite speed, which is a characteristic feature of degenerate diffusion).

\subsection{Error analysis}
In Section \ref{Existence}, we proved convergence of the semi-discretized solution $u_h$  to the unique weak solution $u$ as the time-step $h\to 0$, but without error estimates.
Therefore, any theoretical convergence result $U\to u$  with a priori error estimates as $\Delta x,\Delta t\to 0$ for
the fully discretized problem is certainly beyond the scope of this paper.
However, we wish to present instead a numerical investigation of the convergence orders.

To compare the numerical solution we need an analytical one. To the best of our knowledge this is not possible in the two-component case.
In the scalar case (single component) and for the specific isotherm $b(u)=|u|^{p-1}u$ with Freundlich exponent $p\in (0,1)$, the change of variables $z=b(u)$ turns the homogeneous equation $\partial_t b(u)=\Delta u$ into the celebrated Porous Medium Equation (PME, in short)
$$
\partial_t z =\Delta( |z|^{m-1}z)\hspace{1cm}\mbox{with }m=\frac{1}{p}>1.
$$
As is well known \cite{Va07}, the Cauchy problem  for PME is well posed in the whole space for suitable initial data, and the problem exhibits finite speed of propagation, i.e. for any compactly supported initial data $z^0(x)$ the solution $z(t,.)$ stays compactly supported for all times, and the supports expand with finite speed.
More importantly, the Zel'dovich-Kompaneets-Barenblatt profiles
$$
(t,x)\mapsto \mathcal{Z}(C;t_0;x_0;t,x):=(t+t_0)^{-\alpha}\left(C-k|x-x_0|^2(t+t_0)^{-2\beta}\right)_+^{\frac{1}{m-1}}\geq 0
$$
yield a family of explicit solutions for $t>-t_0$ and $x\in \R^d$. Here $C>0$ is a free parameter corresponding to the $L^1(\R^d)$ mass (which is preserved along evolution), $t_0\in \R$ and $x_0\in \R^d$ translate the invariance under shifts, and
$$
\alpha=\frac{d}{d(m-1)+2},\qquad \beta=\frac{\alpha}{d},\qquad  k=\frac{\alpha(m-1)}{2md}
$$
are universal constants depending on the exponent $m>1$ and the space dimension $d\geq 1$ only. Moving back to $u=|z|^{m-1}z$ yields explicit solutions $(t,x)\mapsto \mathcal{U}(C;t_0;x_0;t,x)$ to the original equation, which we use to compute errors.
The results from the previous sections, and in particular the \emph{a priori} energy estimates, suggest that the implicit scheme given by \eqref{discrete_approx-min-problem}-\eqref{eq:discrete_energy_functional} should be stable unconditionally with respect to $\Delta t$, the stability occurring in the relevant energy spaces corresponding to Theorem~\ref{main}.
For our tests we always chose a fixed computation time $T=.5$, and fix $\Delta t=2 \Delta x$ for convenience. The other parameters are adjusted so that the ZKB solutions stay supported inside the numerical domain $\Omega=(a,b)$, so that these free-boundary solutions correspond to the unique solution to the Cauchy problem in $\Omega$ with zero boundary conditions at least for $t\leq T$.

Note that the larger the diffusion exponent $m>1$, the more degenerate the equation. The speed of propagation becomes infinite in the limit $\frac{1}{p}=m\downarrow 1$ since the PME $\partial_t z=\Delta z^m$ then formally converges to the heat equation, which has infinite speed of propagation as all uniformly parabolic equations.
This is also captured by our scheme, but due to the lack of space we do not present the corresponding simulations here.
The $L^2(Q_T)$ errors are shown in Fig.~\ref{fig:error} as a function of  $\Delta x\to 0$ for various values of $m>1$, suggesting algebraic convergence: 
the convergence seems to be affected only through a prefactor $C_m$ in $e_2\lesssim C_m|\Delta x|^r$, but the rate $r$ appears to be independent of $m$.
\begin{figure}
\begin{center}
 \def\svgwidth{1.\textwidth}
\centering
  \includegraphics[width=\textwidth]{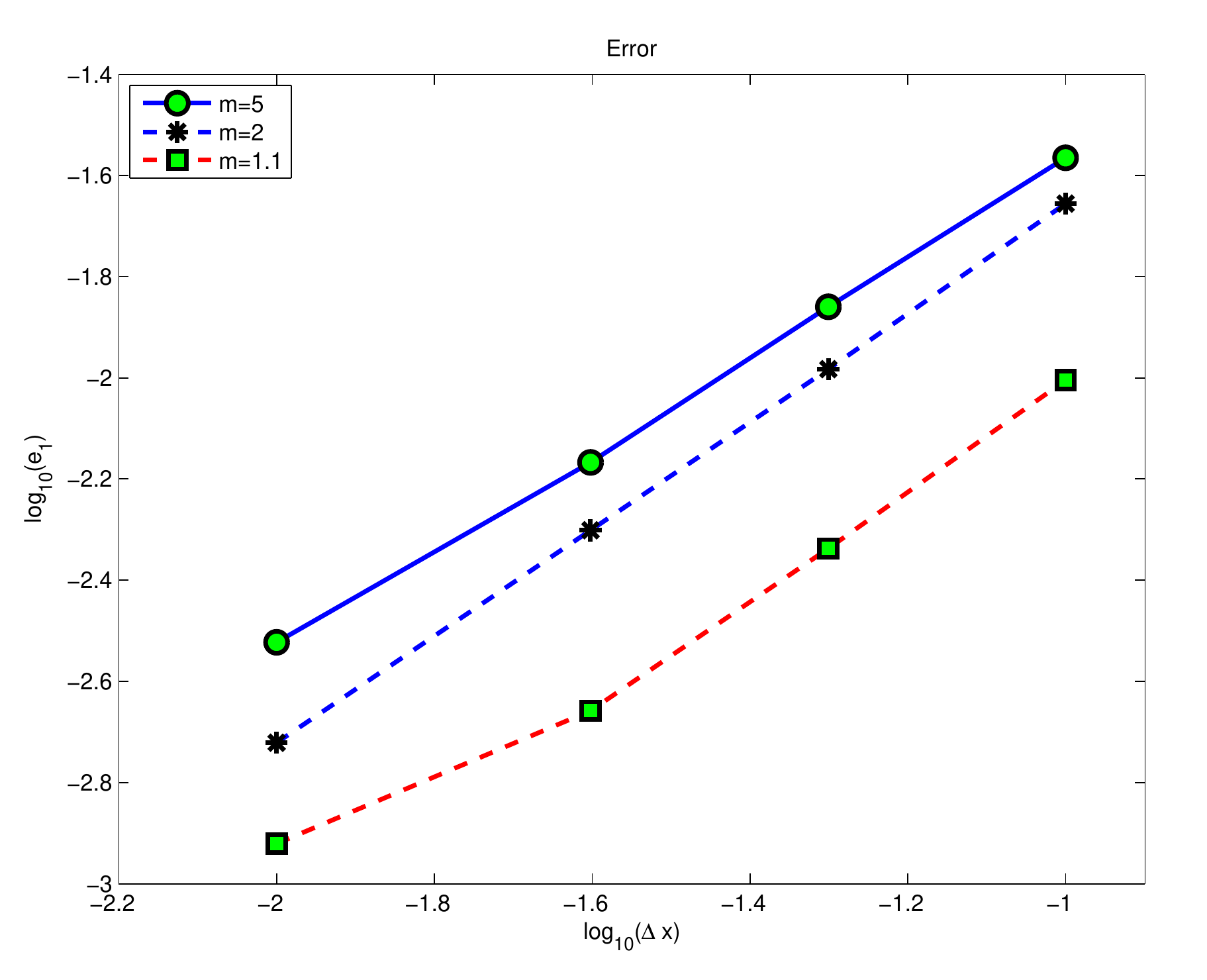}
 \end{center}
  \caption{Logarithmic plot of the $L^2(Q_T)$ errors $e_2$ as a function of $\Delta x$, with $\Delta t=2\Delta x$.}
 \label{fig:error}
 \end{figure}

\subsection*{Acknowledgments}
This work was partially supported by FCT/Portugal through the project UID/MAT/04459/2013 and by the UT Austin|Portugal CoLab project.
FB and LM were supported by the Portuguese National Science Foundation through through FCT
fellowships SFRH/BPD/ 33962/2009 and SFRH/BPD/88207/2012.

\end{document}